%% file: lcd.tex
\documentclass[11pt]{amsart}
\usepackage[utf8]{inputenc}
\usepackage{kantlipsum} 
\usepackage{comment}
\calclayout
\newif\ifdraft
\draftfalse
\input{macros.tex}

\newtheorem{set-up}[equation]{Set-up}

\usepackage{tikz-3dplot}
\definecolor{gray(x11gray)}{rgb}{0.75, 0.75, 0.75}
\definecolor{aliceblue}{rgb}{0.94, 0.97, 1.0}
\usepackage{float}
\usepackage[normalem]{ulem}

\begin{document}

\vspace{\baselineskip}

\title[On the local cohomology of secant varieties]{On the local cohomology of secant varieties}

\author[S.~Olano]{Sebasti\'an~Olano}
\address{Department of Mathematics, University of Toronto, 40 St. George St., Toronto, Ontario 
Canada, M5S 2E4}
\email{{\tt seolano@math.toronto.edu}}

\author[D.~Raychaudhury]{Debaditya Raychaudhury}
\address{Department of Mathematics, University of Arizona, 617 N. Santa Rita Ave., Tucson, Arizona 85721,
USA\linebreak
Present address: Department of Mathematics and Statistics, University of New Mexico, Albuquerque, NM 87131, USA}
\email{rcdeba@gmail.com}

\thanks{}

\subjclass[2020]{14J17, 14N07} 
\keywords{Secant varieties, Hodge filtration on local cohomology and local cohomological dimension}


\maketitle

\vspace{-20pt}

\begin{abstract}
Given a sufficiently positive embedding $X\subset\mathbb{P}^N$ of a smooth projective variety $X$, we consider its secant variety $\Sigma$ that comes equipped with the embedding $\Sigma\subset\mathbb{P}^N$ by its construction. In this article, we determine the local cohomological dimension $\textrm{lcd}(\mathbb{P}^N,\Sigma)$ of this embedding, as well as the generation level of the Hodge filtration on the topmost non-vanishing local cohomology module $\mathcal{H}^{q}_{\Sigma}(\mathcal{O}_{\mathbb{P}^N})$, i.e., when $q=\textrm{lcd}(\mathbb{P}^N,\Sigma)$. Additionally, we show that $\Sigma$ has quotient singularities (in which case the equality $\textrm{lcd}(\mathbb{P}^N,\Sigma)=\textrm{codim}_{\mathbb{P}^N}(\Sigma)$ is known to hold) if and only if $X\cong\mathbb{P}^1$. We also provide a complete classification of $(X,L)$ for which $\Sigma$ has ($\mathbb{Q}$-)Gorentein singularities. As a consequence, we deduce that if $\Sigma$ is a local complete intersection, then either $X$ is isomorphic to $\mathbb{P}^1$, or an elliptic curve.
\end{abstract}


\section{Introduction}

Let $X\subset\mathbb{P}^N:=\mathbb{P}(H^0(L))$ be a smooth projective variety of dimension $n$, embedded by the complete linear series of a very ample line bundle $L$. The {\it secant variety} $\Sigma:=\Sigma(X,L)$ of $X\subset\mathbb{P}^N$ is defined as the Zariski closure of the union of 2-secant lines to $X$ in $\mathbb{P}^N$, i.e.,
$$\Sigma:=\overline{\cup_{x_1,x_2\in X, x_1\neq x_2}\langle x_1,x_2\rangle}\subseteq\mathbb{P}^N.$$
Secant (and higher secant) varieties have a ubiquitous presence in classical algebraic geometry. The dimension of these varieties (or in other words, theirs {\it defectiveness}), their defining equations, and syzygies are topics of great interest that have attracted the attention of algebraic geometers for a long time, see \cite{CC, CR,ENP,Rai,SV,Ver,V2,V1, Zak} and the references therein. The research on these topics dates back more than a hundred years (see for e.g. \cite{sev}) and found important recent applications in several other areas such as tensor geometry, algebraic statistics, and complexity theory (\cite{Lan, LW, SS}).

\vspace{5pt}

It is well-known that if the embedding line bundle $L$ is sufficiently positive, then $\Sigma$ has the expected dimension $2n+1$. It is a natural question to ask: {\it how bad are the singularities of $\Sigma$ when $L$ is sufficiently positive?} In order to carry out a detailed study of this question, perhaps the first agenda that one might be interested in pursuing is to understand when these varieties are normal. The normality of secant varieties has been established by Ullery in \cite{Ull16}. Immediately after her work, Chou-Song in \cite{CS18} showed that under the positivity assumption of Ullery, $\Sigma$ has Du Bois singularities and completely characterized the cases when the singularities of $\Sigma$ are rational. More recently, the question of when these varieties have higher Du Bois or higher rational singularities has been addressed in \cite{ORS}.

\vspace{5pt}

Observe that a secant variety $\Sigma$ is endowed with an embedding $\Sigma\subseteq \mathbb{P}^N$ by its construction. It is then natural to study the local cohomology of this embedding, which gives us further information about the singularities of $\Sigma$. In order to state our results, we first introduce some notation.

\vspace{5pt}

Given an embedding $Z\subset W$ of a variety $Z$ inside a smooth variety $W$, the {\it local cohomological dimension} $\textrm{lcd}(W,Z)$ is an invariant naturally associated with this embedding. It is defined through the local cohomology sheaves $\mathcal{H}_Z^q(\mathcal{O}_W)$ as 
$$\textrm{lcd}(W,Z):=\textrm{max}\left\{q\mid \mathcal{H}_Z^q(\mathcal{O}_W)\neq 0\right\}.$$
It is well-known that if $Z$ is smooth, or more generally a local complete intersection (henceforth we will abbreviate this
as {\it lci}), we have  $\textrm{lcd}(W,Z)=\textrm{codim}_W(Z)$. This fact allows us to regard this invariant as a measure of the singularities of $Z$. 

\vspace{5pt}

Moreover, it turns out that the sheaves $\mathcal{H}_Z^q(\mathcal{O}_W)$ have the structure of a filtered regular, holonomic $\mathcal{D}_W$-module underlying a mixed Hodge module on $W$ with support in $Z$. In particular, they come equipped with the Hodge filtration $F_{\bullet}\mathcal{H}_Z^q(\mathcal{O}_W)$. For each index $q$, associated to this Hodge filtration $F_{\bullet}\mathcal{H}_Z^q(\mathcal{O}_W)$, there is an invariant known as the {\it generation level} $\textrm{gl}(F_{\bullet}\mathcal{H}_Z^q(\mathcal{O}_W))$ which essentially determines the least number of terms $F_k\mathcal{H}_Z^q(\mathcal{O}_W)$ required to completely determine the filtration up to the action of the differential operators. If $Z$ is smooth with $c:=\textrm{codim}_W(Z)=\textrm{lcd}(W,Z)$, then we have $\textrm{gl}(F_{\bullet}\mathcal{H}_Z^c(\mathcal{O}_W))=0$.

\vspace{5pt}

The purpose of this work is to completely determine the quantities $$\mathrm{lcd}(\mathbb{P}^N,\Sigma)\textrm{ and }\mathrm{gl}(F_{\bullet}\mathcal{H}^{\mathrm{lcd}(\mathbb{P}^N,\Sigma)}_{{\Sigma}}(\mathcal{O}_{\mathbb{P}^N}))$$ when $L$ is sufficiently positive. We describe the positivity of the line bundle $L$ in terms of an integer $p$, via a condition called the $(Q_p)$-property described in \definitionref{qp}. It was shown in \cite{ORS} that there are functions $f(p,n)$ and $g(l,p,n)$ such that the pluri-adjoint linear series $lK_X+dA+B$, where $A$ is a very ample and $B$ is a nef line bundle, satisfies $(Q_p)$-property if  $l\geq f(p,n)$ and $d\geq g(l,p,n)$. In particular, $lK_X+dA+B$ satisfies $(Q_p)$-property for all $p$ if $d\gg l\gg 0$. We require another important invariant in order to state the main result. Define $$\nu(X):=\max\left\{i\mid 0\leq i\leq n-1 \textrm{ and } H^j(\mathcal{O}_X)=0\,\textrm{ for all }\, 1\leq j\leq i\right\},$$ 
with the convention that $\nu(X)=0$ if $H^1(\mathcal{O}_X)\neq 0$ or if $n=1$, i.e. when the above set is empty. Set $q_X=N-n=\textrm{codim}_{\mathbb{P}^N}(X)$. Our main result is as follows:

\begin{intro-theorem}\label{mainlcd}
Assume $L$ satisfies $(Q_n)$-property (equivalently, assume $L$ satisfies $(Q_p)$-property for all $p$) and $\Sigma\neq\mathbb{P}^N$. Then $$\mathrm{lcd}(\mathbb{P}^N,\Sigma)=\begin{cases} 
      q_X-2 & \textrm{if }\, \nu(X)=0; \\
      q_X-3 & \textrm{ otherwise}.
   \end{cases}$$
Moreover, the following statements hold:
\begin{enumerate}
    \item If $\nu(X)=0$, then $$\mathrm{gl}(F_{\bullet}\mathcal{H}^{q_X-2}_{{\Sigma}}(\mathcal{O}_{\mathbb{P}^N}))=\begin{cases} 
      0 & \textrm{if }\, (X,L)\cong(\mathbb{P}^1,\mathcal{O}_{\mathbb{P}^1}(d))\textrm{ with } d\geq 4; \\
      1 & \textrm{ otherwise}.
   \end{cases}$$
    \item If $\nu(X)\geq 1$, then $$\mathrm{gl}(F_{\bullet}\mathcal{H}^{q_X-3}_{{\Sigma}}(\mathcal{O}_{\mathbb{P}^N}))=\begin{cases} 
      1 & \textrm{if }\, H^2(\mathcal{O}_X)=0; \\
      2 & \textrm{ otherwise}.
   \end{cases}$$
\end{enumerate}
\end{intro-theorem}

It is useful to compare the result above to the case when $Z \subset W$ is a reduced hypersurface, where the generation level has been related to an important invariant. More precisely, in this case, the {\it minimal exponent} $\widetilde{\alpha_Z}$, which by definition is the negative of the largest root of its reduced Bernstein-Sato polynomial, provides an upper bound on $\textrm{gl}(F_{\bullet}\mathcal{H}_Z^1(\mathcal{O}_W))$. This exponent also detects the maximum $p$ for which the singularities of $Z$ are $p$-Du Bois or $p$-rational. To highlight our contribution, let us list what we can conclude for the secant variety of a quartic rational normal curve, which is a hypersurface, solely through its minimal exponent:

\begin{example}\label{ex1}
Let $\mathbb{P}^1\hookrightarrow\mathbb{P}^4$ be a quartic rational normal curve, and let $Z_0, Z_1, Z_2, Z_3, Z_4$ be the coordinates of $\mathbb{P}^4$. It is well-known that its secant variety is a cubic, which is the determinantal variety associated to the Catalecticant (or Hankel) matrix
$$\begin{pmatrix}
    Z_0 & Z_1 & Z_2\\
    Z_1 & Z_2 & Z_3\\
    Z_2 & Z_3 & Z_4
\end{pmatrix}.$$
The Bernstein-Sato polynomial is given by 
\begin{equation}\label{rbs}
    b_{\Sigma}(s)=(s+1)\left(s+\frac{3}{2}\right).
\end{equation}
In particular, $\widetilde{\alpha_{\Sigma}}=1.5$. 
Thus $\textrm{gl}(F_{\bullet}\mathcal{H}_{\Sigma}^1(\mathcal{O}_{\mathbb{P}^N}))\leq 1$ by \cite[Theorem A]{MP'}. Combining \cite{MPOW} and \cite{jksy22}, it also follows that the singularities of $\Sigma$ are rational but not $1$-Du Bois. Of course, this is in accordance with the results of \cite{ORS}, see \theoremref{previous}.
\end{example}

In view of \theoremref{mainlcd}, we see that the upper bound on $\textrm{gl}(F_{\bullet}\mathcal{H}_{\Sigma}^1(\mathcal{O}_{\mathbb{P}^N}))$ derived from the minimal exponent in the above example is not optimal. It is also important to note that calculating the local cohomological dimension and the generation levels is generally more challenging when the variety is not a locally complete intersection, as is the case for secant varieties, which are typically not even Cohen-Macaulay (see \theoremref{previous} and \corollaryref{corf}).

\vspace{5pt}

We remark an additional consequence of \theoremref{mainlcd}. For an embedded variety $Z\subset W$ inside smooth $W$, it is known that the ideal depth-lcd pattern $\textrm{depth}(\mathcal{O}_Z)\geq k\implies \textrm{lcd}(W,Z)\leq \dim W-k$ does not hold in general if $k\geq 4$, see \cite[Example 2.11]{DT}. It was shown by Chou-Song that $\mathrm{depth}(\mathcal{O}_{\Sigma})=n+2+\nu(X)$ under the assumptions of \theoremref{mainlcd} (c.f. \theoremref{depth} for a more precise statement). In view of this, it is interesting to note that \theoremref{mainlcd} shows that when $L$ is sufficiently positive, we have $$\mathrm{lcd}(\mathbb{P}^N,\Sigma)>N-\textrm{depth}(\mathcal{O}_{\Sigma})=q_X-\nu(X)-2$$ as soon as $\nu(X)\geq 2$. 

\vspace{5pt}

Recall that any complex variety $Z$ comes equipped with the Du Bois complex $\underline{\Omega}_Z^{\bullet}$ which is an object in the bounded derived category of filtered complexes (\cite{dubois81}). The associated graded objects $$\underline{\Omega}_Z^p:=\textrm{Gr}_F^p(\underline{\Omega}_Z^{\bullet})[p]$$ are objects in the derived category of coherent sheaves. Writing ${\bf D}_Z(\underline{\Omega}_Z^p)$ to be the Grothendieck dual of $\underline{\Omega}_Z^p$ (see Sect. \ref{sechodge} for the definition), we obtain a vanishing result as a consequence of \theoremref{mainlcd} as  explained in \cite[Theorem 2.16]{PS}:

\begin{intro-corollary}\label{kodaira}
Assume $L$ satisfies $(Q_n)$-property and $\Sigma\neq\mathbb{P}^N$. Let $\mathcal{L}$ be an ample line bundle on $\Sigma$. 
\begin{enumerate}
    \item We have $${\bf H}^q({\bf D}_{\Sigma}(\underline{\Omega}_{\Sigma}^p)\otimes\mathcal{L})=0\textrm{ when }q-p>\begin{cases}
    n-1\textrm{ if }\nu(X)=0;\\
    n-2\textrm{ otherwise},
\end{cases}$$
or equivalently $${\bf H}^q(\underline{\Omega}_{\Sigma}^p\otimes\mathcal{L}^{-1})=0\textrm{ when }p+q<\begin{cases}
    n+2\textrm{ if }\nu(X)=0;\\
    n+3\textrm{ otherwise}.
\end{cases}$$
\item In particular, $H^q(\Omega_{\Sigma}^{[p]}\otimes \mathcal{L}^{-1})=0$ if one of the following holds:
\begin{itemize}
    \item $\nu(X)=0$, $p=0$ and $q<n+2$; or
    \item $\nu(X)\geq 1$, $H^k(\mathcal{O}_X)=0$ for all $1\leq k\leq p$, and $p+q<n+3$.
\end{itemize}
\end{enumerate}
\end{intro-corollary}

Analogous Kodaira-Akizuki-Nakano type vanishings were established in \cite[Corollary E]{ORS}. We further obtain the following 

\begin{intro-corollary}\label{corb}
Assume $L$ satisfies $(Q_n)$-property and $\Sigma\neq\mathbb{P}^N$. Then the following are equivalent:
\begin{enumerate}
    \item $\mathbb{Q}_{\Sigma}[2n+1]$ is perverse.
    \item Either $n=1$; or $n=2$ and $H^1(\mathcal{O}_X)=0$.
\end{enumerate}
\end{intro-corollary}
The above \corollaryref{corb} is an immediate consequence of \theoremref{mainlcd}, and the fact that if $Z \subset W$ is an embedding inside smooth $W$, $\mathbb{Q}_Z[\dim Z]$ is perverse if and only if $\textrm{lcd}(W,Z) = \textrm{codim}_W(Z)$ (c.f. \propositionref{per} which also extends \cite[Corollary 11.22]{MP} in view of \eqref{chain}). Now, given an embedding $Z\subset W$ inside smooth $W$, there are two special instances where this equality 
holds:
\begin{itemize}
\item[(a)] when $Z$ has quotient singularities (see \eqref{chain}, also \cite[Corollary 11.22]{MP} for a direct proof),
\item[(b)] when $Z$ is lci (in this case $\mathbb{Q}_Z[\dim Z]$ is perverse by \cite{BBD}).
\end{itemize}
Thus, it is useful to understand when $\Sigma$ has quotient singularities, and when $\Sigma$ is lci\footnote{We are very grateful to Mihnea Popa for suggesting these questions to us.}.

\vspace{5pt}
We answer the first question below:

\begin{intro-theorem}\label{mainquot}
Assume $L$ satisfies $(Q_1)$-property. Then $\Sigma$ has quotient singularities if and only if $(X,L)\cong (\mathbb{P}^1,\mathcal{O}_{\mathbb{P}^1}(d))$ with $d\geq 3$.
\end{intro-theorem}

A natural object associated to a variety $Z$ is $\mathbb{Q}_{Z}^H[\dim Z]$ which lives in the bounded derived category of mixed Hodge modules $D^b(\textrm{MHM}(Z))$. Denoting the intersection complex Hodge module as $\textrm{IC}_{Z}\mathbb{Q}^H$, there is a natural map $$\mathbb{Q}_{Z}^H[\dim Z]\to \textrm{IC}_{Z}\mathbb{Q}^H.$$ Now, $Z$ is called {\it rational homology manifold} or {\it rationally smooth} if the map above is an isomorphism. It is well-known that if $Z$ has quotient singularities then $Z$ is rationally smooth. In general, we have the following chain of implications:
\begin{equation}\label{chain}
    \textrm{$Z$ has quotient singularities $\implies$ $Z$ is rationally smooth $\implies$ $\mathbb{Q}_Z[\dim Z]$ is perverse.}
\end{equation}
Observe that by \theoremref{mainquot}, secant varieties of rational normal curves of degree $\geq 3$ are rationally smooth thanks to \eqref{chain} (for quartic rational normal curves, the fact that their secant varieties are rationally smooth can also be seen from their reduced Bernstein-Sato polynomial since it has no integer root by \eqref{rbs}). In fact, it is not hard to show that when $L$ satisfies $(Q_1)$-property, $\Sigma$ is rationally smooth if and only if $X\cong\mathbb{P}^1$. However, a formal proof of this fact will appear elsewhere.

\vspace{5pt}

Using \theoremref{previous}, \theoremref{mainlcd} and \theoremref{mainquot} (we also use \cite[Proposition 4.2(2)]{SVV}), we immediately deduce the following
\begin{intro-corollary}\label{cord}
Assume $L$ satisfies $(Q_n)$-property and $\Sigma\neq\mathbb{P}^N$. The following are equivalent:
\begin{enumerate}
    \item $\Sigma$ has quotient singularities,
    \item the singularities of $\Sigma$ are pre-$1$-rational,
    \item the singularities of $\Sigma$ are pre-$p$-rational for all $p$,
    \item $\mathrm{gl}(F_{\bullet}\mathcal{H}^{\mathrm{lcd}(\mathbb{P}^N,\Sigma)}_{{\Sigma}}(\mathcal{O}_{\mathbb{P}^N}))=0$,
    \item $(X,L)\cong (\mathbb{P}^1,\mathcal{O}_{\mathbb{P}^1}(d))$ with $d\geq 4$.
\end{enumerate}
\end{intro-corollary}
When $X$ is a smooth curve of genus $g$, $L$ is a line bundle of degree $\geq 2g+3$, and $\Sigma\neq\mathbb{P}^N$, any of the conditions (1)-(5) in the above result is equivalent to another set of equivalent conditions involving the Betti numbers and the regularity of $\Sigma$, given by \cite[Theorem 1.1]{CK} (with $q=2$). Furthermore, these conditions are also equivalent to $\Sigma$ being a Fano variety with log terminal singularities by \cite[Theorem 1.1]{ENP}, see \corollaryref{ratnorm} for the precise statement.

\vspace{5pt}

Next we consider the question which asks whether $\Sigma$ is lci or not. Recall that lci varieties are Gorenstein, and we provide below a classification of $(X,L)$ for which $\Sigma$ is ($\mathbb{Q}$-)Gorenstein: 
\begin{intro-theorem}\label{maingor}
Assume $L$ satisfies the following:
\begin{itemize}
    \item When $n=1$, $\mathrm{deg}(L)\geq 2g+3$ where $g$ is the genus.
    \item When $n\geq 2$, $L=K_X+(2n+2)A+B$ with $A$ very ample and $B$ nef line bundles.
\end{itemize}
Then:
\begin{enumerate}
    \item $\Sigma$ is $\mathbb{Q}$-Gorenstein if and only if $(X,L)$ is one of the following:
    \begin{itemize}
        \item $(\mathbb{P}^1,\mathcal{O}_{\mathbb{P}^1}(d))$ with $d\geq 3$,
        \item $(E,L)$ where $E$ is an elliptic curve, $\mathrm{deg}(L)\geq 5$,
        \item $(\mathbb{P}^2,\mathcal{O}_{\mathbb{P}^2}(6))$,
        \item $(\mathbb{P}^1\times\mathbb{P}^1,\mathcal{O}_{\mathbb{P}^1\times\mathbb{P}^1}(4,4))$,
        \item $(\mathbb{P}^3,\mathcal{O}_{\mathbb{P}^3}(4))$.
    \end{itemize}
    \item $\Sigma$ is Gorenstein if and only if $(X,L)$ is one of the following:
    \begin{itemize}
        \item $(\mathbb{P}^1,\mathcal{O}_{\mathbb{P}^1}(d))$ with $d=3$ or $4$,
        \item $(E,L)$ where $E$ is an elliptic curve, $\mathrm{deg}(L)\geq 5$,
        \item $(\mathbb{P}^3,\mathcal{O}_{\mathbb{P}^3}(4))$.
    \end{itemize}
\end{enumerate}
\end{intro-theorem}

We deduce the following:

\begin{intro-corollary}\label{corf}
Assume $L$ satisfies the assumptions of \theoremref{maingor} and $\Sigma\neq\mathbb{P}^N$. If $L$ satisfies $(Q_n)$-property, then $\Sigma$ is lci implies $(X,L)$ is one of the following:
    \begin{itemize}
        \item[(i)] $(\mathbb{P}^1,\mathcal{O}_{\mathbb{P}^1}(4))$,
        \item[(ii)] $(E,L)$ where $E$ is an elliptic curve, $\mathrm{deg}(L)\geq 5$.
    \end{itemize}
Moreover, if (i) holds, or if (ii) holds and $5\leq \mathrm{deg}(L)\leq 6$, then $\Sigma$ is lci (in fact a complete intersection). 
\end{intro-corollary}
We do not know if there is an elliptic normal curve of degree $\geq 7$ whose secant variety $\Sigma$ is lci.

\vspace{5pt}

The organization of this article can be summarized as follows: Sect. \ref{sechodge} is devoted to providing the necessary preliminaries. We describe the geometry of secant varieties in Sect. \ref{secgeom} and prove \theoremref{mainquot} and \theoremref{maingor}. We prove the main technical results in Sect. \ref{secheart} which we use in Sect. \ref{secfinal} to prove \theoremref{mainlcd}.

\vspace{5pt}

We work over the field of complex numbers $\mathbb{C}$. A {\it variety} is an integral separated scheme of finite type over $\mathbb{C}$. We use the additive and multiplicative notation for line bundles interchangeably, and the notation ``$=_{\mathbb{Q}}$'' is used for $\mathbb{Q}$-linear equivalence of divisors (or line bundles). \\

\noindent{\bf Acknowledgements.} We are very grateful to Mircea Musta\c{t}\u{a} and Mihnea Popa for valuable comments on an earlier draft of this work. We are indebted to Lei Song for his detailed comments on the earlier draft, and for conversations at various stages of this work. The second author also expresses his gratitude to Angelo Felice Lopez for patiently answering his questions and clarifying his doubts. We thank the referee for their comments and corrections that improved the exposition.

\section{Preliminaries and an overview of local cohomology}\label{sechodge}
This section is devoted to supplying the necessary preliminaries. Let us first introduce the notation for the Grothendieck duality functor: given a variety $Z$ with dualizing complex $\omega_Z^{\bullet}$, we set $${\bf D}_Z(-):={\bf R}\mathcal{H}\textit{om}_{\mathcal{O}_Z}(-,\omega_Z^{\bullet})[-\dim Z].$$
We mention an useful fact here: if $Z\subset W$ is a subvariety of codimension $c$, we have an isomorphism ${\bf D}_{W}(-)\cong{\bf D}_{Z}(-)[-c]$ for complexes of $\mathcal{O}_Z$-modules. The complexes $\underline{\Omega}_Z^p$ and ${\bf D}_Z(\underline{\Omega}_Z^p)$ encode various information regarding the singularities of $Z$. To compute them, it is often useful to work with a specific kind of log resolution of $Z$ that we define below:

\begin{definition}\label{defsl}
    A proper morphism $\mu:\Tilde{Z}\to Z$ is called a {\it strong log resolution} if $\mu$ is an isomorphism over $Z_{\textrm{sm}}:=Z\backslash Z_{\textrm{sing}}$, and $\mu^{-1}(Z_{\textrm{sing}})_{\textrm{red}}$ is a divisor with simple normal crossings.  
\end{definition}

In our study of the local cohomology of secant varieties $\Sigma$, we will use an explicit strong log resolution of $\Sigma$ that we describe in Sect. \ref{secgeom}.

\subsection{Local cohomological dimension and Hodge filtrations} 
Let $Z$ be a proper closed subscheme of a smooth variety $W$. For a quasi-coherent $\mathcal{O}_W$ module $\mathcal{M}$ and $q\in\mathbb{N}$, the {\it $q$-th local cohomology sheaf} $\mathcal{H}^q_{Z}(\mathcal{M})$ is by definition the $q$-th derived functor of $\underline{\Gamma}_Z(-)$ given by the subsheaf of local sections with support in $Z$. We refer to \cite{Har'} for more details on local cohomology. 

\smallskip

It is well-known that $\textrm{codim}_W(Z)=\min\left\{q\mid \mathcal{H}^q_{Z}(\mathcal{O}_W)\neq 0 \right\}$. However, the highest $q$ for which the $q$-th local cohomology sheaf $\mathcal{H}^q_{Z}(\mathcal{O}_W)\neq 0$ is a more mysterious object:

\begin{definition}\label{deflcd}
The {\it local cohomological dimension} of $Z$ in $W$ is defined as 
$$\textrm{lcd}(W,Z):=\textrm{max}\left\{q\mid \mathcal{H}_Z^q(\mathcal{O}_W)\neq 0\right\}.$$
\end{definition}

This invariant can be characterized alternatively in terms of the {\it de Rham depth} $\textrm{DRD}(Z)$, or in terms of the {\it rectified $\mathbb{Q}$-homological depth} $\textrm{RHD}_{\mathbb{Q}}(Z^{\textrm{an}})$ (see \cite{Ogu}, \cite{RSW} for details). One has the equality $\textrm{lcd}(W,Z)=\textrm{codim}_W(Z)$ if $Z\subset W$ is smooth. In fact, we have the following 


\begin{proposition}\label{per}
The shifted constant sheaf $\mathbb{Q}_Z[\dim Z]$ is perverse if and only if $\mathrm{lcd}(W,Z)=\mathrm{codim}_W(Z)$. 
\end{proposition}
\begin{proof}
The local cohomological dimension can be described in terms of the perverse cohomology of the constant sheaf, namely, $$\mathrm{lcd}(W,Z) = \dim W - \min\{ j\in\ZZ \mid \ ^{\mathfrak{p}}\cohH^j(\QQ_Z)\neq 0 \},$$ see \cite{RSW}*{Theorem 1} or \cite{bblsz}*{\textsection 3}. This implies that if $\mathbb{Q}_Z[\dim Z]$ is perverse, then $\mathrm{lcd}(W,Z)=\mathrm{codim}_W(Z)$. Moreover, the perverse cohomology of $\mathbb{Q}_Z[\dim Z]$ is concentrated in non-positive degrees (see e.g. \cite{saito90}*{(4.5.6)}). Then, if $\mathrm{lcd}(W,Z)=\mathrm{codim}_W(Z)$, $^{\mathfrak{p}}\cohH^j(\QQ_Z[\dim Z])\neq 0$ only when $j=0$. Also, the previous fact means that we have a map $^{\mathfrak{p}}\cohH^0(\QQ_Z[\dim Z]) \to \QQ_Z[\dim Z] $, whose cone $C$ satisfies $^{\mathfrak{p}}\cohH^j(C)= 0$ for all $j\in\ZZ$, and therefore, $C=0$ \cite{BBD}*{Proposition 1.3.7}.
\end{proof}

If $\mathcal{M}$ has the structure of a left $\mathcal{D}_W$-module, then it turns out that $\mathcal{H}^q_Z(\mathcal{M})$ also inherits such a structure. Now, $\mathcal{O}_W$ is the underlying $\mathcal{D}_W$-module of the trivial Hodge module $\mathbb{Q}_W^H[\dim W]$, whence $\mathcal{H}_Z^q(\mathcal{O}_W)$ has the structure of a left $\mathcal{D}_{W}$-module. Moreover, denoting the embedding $Z\hookrightarrow W$ by $i$, the left $\mathcal{D}_{W}$-module $\mathcal{H}_Z^q(\mathcal{O}_W)$ is nothing but the underlying left $\mathcal{D}_{W}$-module of the mixed Hodge module $\mathcal{H}^q(i_*i^{!}\mathbb{Q}_W[\dim W])$, whence it carries a canonical Hodge filtration $F_{\bullet}\mathcal{H}_Z^q(\mathcal{O}_W)$. Moreover, $F_p\mathcal{H}_Z^q(\mathcal{O}_W)=0$ for all $p<0$ (see \cite[Remark 3.4]{MP}). 

\begin{definition}\label{defgenlevel}
A good filtration $F_{\bullet}\mathcal{M}$ on a left $\mathcal{D}_W$-module is {\it generated at level $k\in\mathbb{Z}$} if $$F_{k+k'}\mathcal{M}=F_{k'}\mathcal{D}_W\cdot F_{k}\mathcal{M}\,\textrm{ for all }\, k'\geq 0,\textrm{ equivalently }F_{m+1}\mathcal{M}=F_1\mathcal{D}_W\cdot F_m\mathcal{M}\textrm{ for all $m\geq k$.}$$ 
 Set $\textrm{gl}(F_{\bullet}\mathcal{M}):=\min\left\{k\mid\textrm{$F_{\bullet}\mathcal{M}$ is generated at level $k$}\right\}$.
\end{definition}
If $Z\subset W$ is smooth with $\textrm{codim}_W(Z)=c$, then the Hodge filtration on $\mathcal{H}_Z^c(\mathcal{O}_W)$ is generated at level $0$. We refer to \cite{MP} for more details on the Hodge filtrations on local cohomology modules, their generation levels and local cohomological dimension.

\section{Geometry of secant varieties}\label{secgeom} Let $L$ be a very ample line bundle on a smooth projective variety $X$ of dimension $n$ that induces the embedding $X\hookrightarrow\mathbb{P}^N:=\mathbb{P}(H^0(L))$. The {\it secant variety} $\Sigma(X,L)$ by definition is the Zariski closure of the union of 2-secant lines of $X$, i.e., we have the commutative diagram:
\begin{equation*}
    \begin{tikzcd}
    X\arrow[rr, hook]\arrow[dr, hook] & & \mathbb{P}^N=\mathbb{P}(H^0(L))\\
    & \Sigma(X,L)\arrow[ur, hook]
\end{tikzcd}
\end{equation*}
We will simply write $\Sigma$ for $\Sigma(X,L)$ to ease the notation.

\vspace{5pt}

We introduce the positivity properties of $L$ that we will require in the sequel. By definition, $L$ is called {\it $k$-very ample} for an integer $k\geq 0$ if the evaluation map of global sections $H^0(L)\to H^0(L\otimes \mathcal{O}_{\xi})$ is surjective for any $0$--dimensional subscheme $\xi$ of length $k+1$. For any $x\in X$, we denote by $\mathcal{I}_x$ the ideal sheaf of $x\in X$. We further set $b_x:F_x\to X$ to be the blow-up of $X$ at $x$, and $E_x$ the exceptional divisor.

\begin{definition}[\cite{ORS}]\label{qp}
Let $p\geq 0$ be an integer, and let $L$ be a 3-very ample line bundle on $X$. Then $L$ is said to satisfy {\it $(Q_p)$-property} if the following conditions are satisfied for all $x\in X$:
\begin{enumerate}
    \item the natural map $\text{Sym}^iH^0(L\otimes \mathcal{I}_x^2)\to H^0(L^{\otimes i}\otimes \mathcal{I}_x^{2i})$ is surjective for all $i\geq 1$,
    \item $b_x^*L(-2E_x)$ is ample, and
    \item $H^i(\Omega^q_{F_x}\otimes b_x^*(jL)(-2jE_x))=0$ for all $i,j\geq 1$, $0\leq q\leq p$.
\end{enumerate}
\end{definition}
It is immediate from the above definition that if $L$ satisfies $(Q_p)$-property, then it satisfies $(Q_k)$-property for all $0\leq k\leq p$. Furthermore, if $L$ satisfies $(Q_n)$-property, then it satisfies $(Q_p)$-property for all $p\geq 0$.

\vspace{5pt}

Throughout this article, we tacitly assume that $L$ is $3$-very ample.

\subsection{Log resolutions of \texorpdfstring{$\Sigma$}{TEXT}} 
Under our assumption of $3$-very ampleness of $L$, we have an explicit log resolution of $\Sigma$ 
coming from \cite{Ver} (see also \cite{Ull16}) that we now describe. 

\smallskip


Let us denote the Hilbert scheme of two points on $X$ by $X^{[2]}$. Recall that $X^{[2]}$ is a smooth projective variety. The universal subscheme $\Phi$ is the incidence variety:
$$\Phi:=\left\{(x,\xi)\in X\times X^{[2]}: x\in\xi\right\}\subset X\times X^{[2]}.$$ 
Clearly $\Phi$ is equipped with two natural projections $q:\Phi\to X$ and $\sigma:\Phi\to X^{[2]}$. Moreover, it turns out that $\Phi\cong \textrm{Bl}_{\Delta}(X\times X)$, i.e., it is isomorphic to the blow-up $X\times X$ along the diagonal $\Delta$. Let $b_{\Delta}:\Phi\cong\textrm{Bl}_{\Delta}(X\times X)\to X\times X$ be the blow-up morphism. Then we have the following commutative diagram:
\begin{equation}\label{diag1}
    \begin{tikzcd}
    & \Phi\arrow[dl, swap, "\sigma"]\arrow[rr, "b_{\Delta}"]\arrow[dr, "q"] & & X\times X\arrow[dl, swap, "p_1"]\\
    X^{[2]} && X
\end{tikzcd}
\end{equation}
Now, the vector bundle $\mathcal{E}_L:=\sigma_*q^*L$ is globally generated since $L$ is very ample. The evaluation map of its global sections
induces $f:\mathbb{P}(\mathcal{E}_L)\to\mathbb{P}(H^0(L))$ which surjects onto the secant variety $\Sigma$. Consequently, we have the surjective map $t:\mathbb{P}(\mathcal{E}_L)\to \Sigma$. The main result is the following: the map $t$ is a log resolution of $\Sigma$; moreover, if $\Sigma\neq\mathbb{P}(H^0(L))$, the map $t:\mathbb{P}(\mathcal{E}_L)\to \Sigma$ is a strong log resolution of $\Sigma$ (for the last part, see {\cite[Corollary 2.7]{ORS}}).

\smallskip

Recall that $F_x\cong\textrm{Bl}_xX$ is the blow-up of $X$ at $x$ with exceptional divisor $E_x$, and we denote the blow-up morphism by $b_x:F_x\cong\textrm{Bl}_xX\to X$. 
The diagram \eqref{diag1} induces the following commutative diagram with Cartesian squares:
\begin{equation}\label{diag2}
\begin{tikzcd}
    \mathbb{P}^{n-1}\cong E_x\arrow[r, hook]\arrow[d] & F_x\arrow[r, hook]\arrow[d, "b_x"] & \Phi\cong\textrm{Bl}_{\Delta}(X\times X)\arrow[d, "b_{\Delta}"]\arrow[dd, bend left=60, "q"]\\
    \left\{(x,x)\right\}\arrow[r, hook]& \left\{x\right\}\times X \arrow[r, hook]\arrow[d] & X\times X\arrow[d, "p_1"]\\
    & \left\{x\right\}\arrow[r, hook] & X
\end{tikzcd}
\end{equation}
We made a slight abuse of notation in the above: we denote the map $F_x\to\left\{x\right\}\times X$ by the same symbol $b_x$ (this is because it sends $y\in F_x$ to $(x,b_x(y))$).

It turns out that $\Phi\cong t^{-1}(X)$. 
To summarize, for any $x\in X$, we have the following diagram with Cartesian squares and surjective vertical arrows:
\begin{equation}\label{ulldiag}
    \begin{tikzcd}
    F_x\arrow[r, hook]\arrow[d] & \Phi\arrow[r, hook]\arrow[d, "q"] & \mathbb{P}(\mathcal{E}_L)\arrow[d, "t"]\arrow[dr, "f"] &\\
    \{x\}\arrow[r, hook] & X\arrow[r, hook] & \Sigma\arrow[r, hook] & \mathbb{P}(H^0(L))
\end{tikzcd}
\end{equation}
We will frequently use the fact that $\dim\Sigma=2n+1$. Also, the map $q:\Phi\to X$ is smooth by \cite[Lemma 2.1]{CS18}. The exact sequence 
\begin{equation}\label{gen1ex}
    0\to q^*\Omega_X^1\to\Omega_{\Phi}^1\to\Omega_{\Phi/X}^1\to 0
\end{equation}
when restricted to $F_x$, yields the following exact sequence 
\begin{equation}\label{filt1}
    0\to\mathcal{O}_{F_x}^{\oplus n}\to\Omega_{\Phi}^1|_{F_x}\to\Omega_{F_x}^1\to 0
\end{equation}
Notice that the above shows $\mathcal{N}_{F_x/\Phi}^*\cong \mathcal{O}_{F_x}^{\oplus n}$. Moreover, \cite[Proof of Lemma 2.3]{Ull16} gives 
\begin{equation}\label{further}
    \mathcal{N}_{\Phi/\mathbb{P}(\mathcal{E}_L)}^*|_{F_x}\cong b_x^*(L)(-2E_x).
\end{equation} 
Observe also that by taking determinants of \eqref{filt1}, we obtain 
\begin{equation}\label{kphi}
    \omega_{\Phi}|_{F_x}\cong \omega_{F_x}
\end{equation}

Let us introduce a few more notation. We set $j_{\Delta}:\Delta\hookrightarrow X\times X$, and $j_{\Delta}':E_{\Delta}\hookrightarrow\Phi$ to be the natural embeddings where $E_{\Delta}$ denotes the exceptional divisor of $b_{\Delta}$.
Observe that \eqref{diag1} also induces the commutative diagram with Cartesian squares:
\begin{equation}\label{diag3}
    \begin{tikzcd}
    \mathbb{P}^{n-1}\cong E_x=F_x\cap E_{\Delta}\arrow[r, hook]\arrow[d] & E_{\Delta}\arrow[r, hook, "j_{\Delta}'"]\arrow[d, "q_{\Delta}"] & \Phi\cong\textrm{Bl}_{\Delta}(X\times X)\arrow[d, "b_{\Delta}"]\arrow[dd, bend left=60, "q"]\\
    \left\{(x,x)\right\}\arrow[r, hook]& \Delta\arrow[r, hook, "j_{\Delta}"]\arrow{dr}{\cong}[swap]{p_0} & X\times X\arrow[d, "p_1"]\\
    & & X
\end{tikzcd}
\end{equation}
where $q_{\Delta}$ is the natural map, and $p_0=p_1\circ j_{\Delta}$. As before, since the map $q_{\Delta}$ is smooth, we have the exact sequence 
\begin{equation}
    0\to q_{\Delta}^*\Omega_{\Delta}^1\to \Omega_{E_{\Delta}}^1\to \Omega_{E_{\Delta}/\Delta}^1\to 0.
\end{equation}
Restricting the above sequence on $E_x$, we obtain the exact sequence
\begin{equation}\label{filt2}
    0\to\mathcal{O}_{E_x}^{\oplus n}\to\Omega^1_{E_{\Delta}}|_{E_x}\to \Omega_{E_x}^1\to 0.
\end{equation}
Notice that it is split (since $\textrm{Ext}^1(\Omega_{E_x}^1,\mathcal{O}_{E_x}^{\oplus n})=0$) whence we get the splitting 
\begin{equation}\label{another}
    \Omega_{E_{\Delta}|_{E_x}}^1\cong \mathcal{O}_{E_x}^{\oplus n}\oplus \Omega_{E_x}^1.
\end{equation}

We recall a result here that will be useful in the sequel:

\begin{proposition}[{\cite[Proposition 3.3]{ORS}}, see {\cite[Proof of Proposition 3.2]{CS18}} when $p=0$]\label{rivanishing}
Let $p\geq 0$, and assume $L$ satisfies $(Q_p)$-property. Then $$R^it_*\Omega_{\mathbb{P}(\mathcal{E}_L)}^{k}(\log\Phi)(-\Phi)=0\textrm{ for all }i\geq 1,\, 0\leq k\leq p.$$
\end{proposition}
To prove the above, the following exact sequence was used which holds for any $j\geq 0$:
\begin{equation}\label{prev1}
    0\to \Omega^j_{\mathbb{P}(\mathcal{E}_L)}(\log\Phi)(-\Phi)\to \Omega^j_{\mathbb{P}(\mathcal{E}_L)}\to \Omega^j_{\Phi}\to 0.
\end{equation}
Moreover, we also have the exact sequence which holds for any $j\geq 1$:
\begin{equation}\label{prev2}
    0\to \Omega^j_{\mathbb{P}(\mathcal{E}_L)}\to \Omega^j_{\mathbb{P}(\mathcal{E}_L)}(\log\Phi)\to \Omega^{j-1}_{\Phi}\to 0. 
\end{equation}
Further, when $j\geq 1$, the above two sequences fit together in the commutative diagram below with exact rows and columns which will also be useful for us:
\begin{equation}\label{omega}
\begin{tikzcd}
    & & 0\arrow[d] & 0\arrow[d] &\\
    0\arrow[r] & \Omega^j_{\mathbb{P}(\mathcal{E}_L)}(\log\Phi)(-\Phi)\arrow[r]\arrow[d, equal] & \Omega^j_{\mathbb{P}(\mathcal{E}_L)}\arrow[r]\arrow[d] & \Omega^j_{\Phi}\arrow[r]\arrow[d] & 0\\
    0\arrow[r] & \Omega^j_{\mathbb{P}(\mathcal{E}_L)}(\log\Phi)(-\Phi)\arrow[r] & \Omega^j_{\mathbb{P}(\mathcal{E}_L)}(\log\Phi)\arrow[r]\arrow[d] & \Omega^j_{\mathbb{P}(\mathcal{E}_L)}(\log\Phi)|_{\Phi}\arrow[r]\arrow[d] & 0\\
    & & \Omega^{j-1}_{\Phi}\arrow[r,equal]\arrow[d] &\Omega^{j-1}_{\Phi}\arrow[d] &\\
    && 0 & 0 &
\end{tikzcd}
\end{equation}

Recall from the Introduction that we have set 
$$\nu(X):=\max\left\{i\mid 0\leq i\leq n-1 \textrm{ and } H^j(\mathcal{O}_X)=0\,\textrm{ for all }\, 1\leq j\leq i\right\}.$$ 
Our convention is that $\nu(X)=0$ if the above set is empty. This invariant is related to the depth of the structure sheaf of $\Sigma$ as follows:
\begin{theorem}[{\cite[Theorem 1.3]{CS18}}]\label{depth} If $L$ satisfies $(Q_0)$-property, then 
    $\mathrm{depth}(\mathcal{O}_{\Sigma})=n+2+\nu(X)$.
\end{theorem}

We finish this brief introduction to secant varieties by recalling the following basic result regarding their singularities:

\begin{theorem}[{\cite{Ull16}}, {\cite{CS18}}, {\cite{ORS}}]\label{previous} The following statements hold:
\begin{enumerate}
    \item If $L$ satisfies $(Q_0)$-property, then:
    \begin{itemize}
        \item[(a)] $\Sigma$ is normal, and its singularities are Du Bois.
        \item[(b)] $\Sigma$ is Cohen-Macaulay if and only if $\nu(X)=0$.
        \item[(c)] $\Sigma$ has weakly rational singularities if and only if $H^n(\mathcal{O}_X)=0$.
        \item[(d)] $\Sigma$ has rational singularities if and only if $H^i(\mathcal{O}_X)=0$ for all $i>0$.
    \end{itemize}
    \item Let $p\geq 1$ and assume $L$ satisfies $(Q_p)$-property. Then:
    \begin{itemize}
        \item[(a)] The singularities of $\Sigma$ are pre-$p$-Du Bois.
        \item[(b)] If $p\leq \lfloor\frac{n}{2}\rfloor$, then the singularities of $\Sigma$ are $p$-Du Bois if and only if $p\leq \nu(X)$.
    \end{itemize}
    \item Assume $L$ satisfies $(Q_1)$-property. Then the singularities of $\Sigma$ are pre-$1$-rational if and only if $X\cong\mathbb{P}^1$.
\end{enumerate}
\end{theorem}

For the definitions of higher Du Bois and higher rational singularities appearing in the statement above, we refer to \cite{SVV} or \cite{ORS}.

\subsection{Secant varieties with special singularities} In this section, we prove \theoremref{mainquot} and \theoremref{maingor}. We start with the first result.

\begin{proof}[Proof of \theoremref{mainquot}] 
First, assume $\Sigma$ has quotient singularities. By \cite[Proposition 4.2(2)]{SVV} the singularities of $\Sigma$ are pre-$1$-rational, whence by \theoremref{previous} (3), we conclude $X\cong\mathbb{P}^1$. To see the converse, we recall a special case of a result from \cite[Proposition 3.5]{Mitch} which says that when $(X,L)\cong (\mathbb{P}^1,\mathcal{O}_{\mathbb{P}^1}(d))$ with $d\geq 3$, $\Sigma$ can be covered by varieties of the form $\mathbb{C}\times \widehat{T}_{\Delta}$ where $\widehat{T}_{\Delta}$ is a normal toric variety of dimension 2. Since 2-dimensional toric varieties are simplicial, it follows that $\Sigma$ in this case has quotient singularities.
\end{proof}
We now aim to prove \theoremref{maingor}. We will use some basic facts from adjunction theory (for a comprehensive literature on this topic, see \cite{BS}). In particular, for the very ample line bundle $A$ on $X$, we use the {\it nef value} of $A$ which is defined as $$\tau(A):=\min \left\{s\in\mathbb{R}\mid K_X+sA\textrm{ is nef}\right\}.$$ We first need an auxiliary 
\begin{proposition}\label{propgor}
Assume $n\geq 2$ and $L$ is 3-very ample. Further assume $\Sigma$ is normal and $\mathbb{Q}$-Gorenstein. Then $$K_X+\left(\frac{n-1}{2}\right)L=_{\mathbb{Q}}0.$$  
\end{proposition}
\begin{proof}
Since $K_{\Sigma}$ is $\mathbb{Q}$-Cartier by assumption, from \eqref{ulldiag} we conclude that 
\begin{equation}\label{k1}
    t^*K_{\Sigma}|_{F_x}=_{\mathbb{Q}}0.
\end{equation}
Now, $(K_{\mathbb{P}(\mathcal{E}_L)}+\Phi)|_{\Phi}=K_{\Phi}$ by adjunction, whence by \eqref{kphi} we obtain 
\begin{equation}\label{k2}
    (K_{\mathbb{P}(\mathcal{E}_L)}+\Phi)|_{F_x}=K_{F_x}.
\end{equation}
To this end, set $a:=a(\Phi;\Sigma,0)$ to be the discrepancy of the $t$-exceptional divisor $\Phi$. Then by definition, we have 
\begin{equation}\label{k3}
    K_{\mathbb{P}(\mathcal{E}_L)}=_{\mathbb{Q}}t^*K_{\Sigma}+a\Phi.
\end{equation}
Combining \eqref{k1}, \eqref{k2} and \eqref{k3} with the fact $K_{F_x}=b_x^*K_X+(n-1)E_x$,  we obtain 
$$(a+1)\Phi|_{F_x}=_\mathbb{Q}b_x^*K_X+(n-1)E_x,$$ which through \eqref{further} simplifies to 
$$b_x^*(K_X+(a+1)L)=_{\mathbb{Q}}(2a+3-n)E_x.$$
Since $n\geq 2$, the above is possible only when both sides are zero, which proves the assertion.
\end{proof}

\begin{proof}[Proof of \theoremref{maingor}] 
Assume $\Sigma$ is $\mathbb{Q}$-Gorenstein. If $n=1$, then $g\leq 1$ by \cite[Remark 5.7]{ENP} whence we assume $n\geq 2$. Recall that $L=K_X+(2n+2)A+B$ by assumption (which is 3-very ample, see for e.g. \cite[Proof of Corollary C]{Ull16}), whence by \propositionref{propgor}, we deduce that 
\begin{equation}\label{k4}
    K_X+(2n-2)A+\frac{n-1}{n+1}B=_{\mathbb{Q}} 0.
\end{equation}
Consequently $\tau(A)\geq 2n-2$, and equality holds only if $B=_{\mathbb{Q}}0$. But it is well-known (and easy to see using Castelnuovo-Mumford regularity) that $\tau(A)\leq n+1$, which gives us $$2n-2\leq\tau(A)\leq n+1\implies n\leq 3.$$
When $n=3$, $\tau(A)=4$ whence $(X,L)\cong (\mathbb{P}^3,\mathcal{O}_{\mathbb{P}^3}(4))$ by \cite[Theorem 7.2.1]{BS} and \propositionref{propgor}. When $n=2$, $2\leq \tau(A)\leq 3$, and note that $K_X+2A$ is not nef and big by \eqref{k4}. Thus, again by \propositionref{propgor}, we see that $X$ is Del Pezzo and $L=_{\mathbb{Q}}-2K_X$. Since numerical equivalence coincides linear equivalence for Fano varieties, we deduce that in fact $L=-2K_X$ and by \cite[Proposition 7.2.2]{BS} one of the following holds:
\begin{itemize}
    \item $(X,L)\cong (\mathbb{P}^2,\mathcal{O}_{\mathbb{P}^2}(6))$,
    \item $(X,L)\cong (\mathbb{P}^1\times\mathbb{P}^1,\mathcal{O}_{\mathbb{P}^1\times\mathbb{P}^1}(4,4))$,
    \item $X\ncong \mathbb{P}^1\times\mathbb{P}^1$, and $X\cong \mathbb{P}(\mathcal{F})$ for a rank 2 bundle $\mathcal{F}$ over a smooth curve $C$.
\end{itemize}
In the third case, we get $0=h^1(\mathcal{O}_X)=h^1(\mathcal{O}_C)$ (the first equality follows from the fact $X$ is Del Pezzo), whence $C\cong\mathbb{P}^1$ and $X$ is a Hirzebruch surface $\mathbb{F}_e$, $e\neq 0$. It follows that $e=1$ (since $X$ is both Del Pezzo and Hirzebruch, and since $X\ncong \mathbb{F}_0$). Recall that $\textrm{Pic}(\mathbb{F}_1)=\mathbb{Z}[C_0]+\mathbb{Z}[F]$ with $[C_0]$ being the class of a section of the structure morphism $\mathbb{F}_1\to\mathbb{P}^1$ with $C_0^2=-1$, and $[F]$ being that of a fiber with $F^2=0$. Also $K_X=-2C_0-3F$. But in this case, $L-K_X=6C_0+9F=6A+B$ which is a contradiction since $rC_0+sF$ is ample (resp. nef) if and only if $r\geq 1, s\geq r+1$ (resp. $r\geq 0, b\geq r$). The converse of (1), and entire (2) now follow from \cite[Theorem 4.4 (G9), (G10) and Theorem 4.6 (Q2), (Q4), (Q5)]{Mitch} (thanks also to the fact that the secant variety $\Sigma$ of an elliptic normal curve of degree $\geq 5$ is Gorenstein by \cite[Remark 5.6]{ENP}).
\end{proof}

\begin{remark}
There are other pairs $(X,L)$, not included in the list given in \theoremref{maingor}, with $L$ 3-very ample but less positive than the requirement of \theoremref{maingor}, whose secant varieties are $(\mathbb{Q}-)$ Gorenstein. For example, according to \cite[Theorem 4.4 (G9)]{Mitch}, the third Veronese embedding of $\mathbb{P}^5$ i.e., $\mathbb{P}^5\hookrightarrow\mathbb{P}^{55}$ embedded by $|\mathcal{O}_{\mathbb{P}^5}(3)|$, has Gorenstein secant variety $\Sigma$ (but the pair $(\mathbb{P}^5,\mathcal{O}_{\mathbb{P}^5}(3))$ of course satisfies the constraint imposed by \propositionref{propgor}).
\end{remark}

\section{Local freeness and ranks of \texorpdfstring{$R^iq_*\Omega_{\Phi}^j$}{TEXT}} \label{secheart}

We continue working with the notation and hypothesis of the previous section. In particular, $X$ is a smooth projective variety of dimension $n$ and $L$ is a 3-very ample line bundle on $X$. 

\subsection{Statement of the key result} In \cite[Lemma 2.2]{CS18}, it was proven that 
\begin{equation}\label{cs}
    R^jq_*\mathcal{O}_{\Phi} \cong H^j(X,\mathcal{O}_X)\otimes \mathcal{O}_X.
\end{equation} 
We also recall:
\begin{lemma}[{\cite[Lemma 2.19]{ORS}}]\label{ext}
    The following statements hold:
    \begin{enumerate}
        \item If $n=1$, then we have the following isomorphisms for all $j$:
        $$R^jq_*\Omega_{\Phi}^1\cong \left[H^j(\mathcal{O}_X)\otimes \Omega_X^1\right]\oplus \left[H^j(\Omega_X^1)\otimes\mathcal{O}_X\right].$$
        \item Assume $n\geq 2$. Then:
        \begin{itemize}
            \item[(i)] $R^jq_*\Omega_{\Phi}^1\cong \left[H^j(\mathcal{O}_X)\otimes \Omega_X^1\right]\oplus \left[H^j(\Omega_X^1)\otimes\mathcal{O}_X\right]$ for all $j\neq 1$;
            \item[(ii)] We have an exact sequence $$0\to \left[H^1(\mathcal{O}_X)\otimes \Omega_X^1\right]\oplus \left[H^1(\Omega_X^1)\otimes \mathcal{O}_X\right]\to R^1q_*\Omega_{\Phi}^1\to\mathcal{O}_X\to 0.$$
        \end{itemize}
    \end{enumerate}
\end{lemma}
Notice that \eqref{cs} and \lemmaref{ext} show the local freeness of $R^iq_*\Omega_{\Phi}^j$ for $j=0,1$ and compute their ranks. 

\smallskip

We aim to extend this to all $j$. To do so, for $j\geq 1$, we set $\mathcal{K}_j$ to be the kernel of the natural surjective map $\Omega_{\Phi}^j|_{F_x}\to \Omega^j_{E_{\Delta}}|_{E_x}$, i.e., it fits into the following exact sequence
$$0\to \mathcal{K}_j\to \Omega_{\Phi}^j|_{F_x}\to \Omega^j_{E_{\Delta}}|_{E_x}\to 0.$$ Next, we construct a filtration $(\mathcal{K}_j, {}_jK^{\bullet})$ on $\mathcal{K}_j$. The main result of this section shows the local freeness of $R^iq_*\Omega_{\Phi}^j$ for all $j$, and inductively computes its rank which we summarize as follows:

\begin{theorem}\label{mainfiltthm}
For any $i,j,k$ and $x\in X$, the following formulae hold:
\begin{equation}\label{recursion}
\arraycolsep=1.4pt\def\arraystretch{2.2}
\begin{array}{c}
h^i(\Omega^j_{\Phi}|_{F_x})=\displaystyle\binom{n}{j-i}h^{i,i}(E_x)+h^i({}_jK^0)\\
h^i({}_jK^k)=h^i({}_jK^{k+1})+\displaystyle\binom{n}{k}\left(h^{i,j-k}(F_x)-h^{i,j-k}(E_x)\right).
\end{array}
\end{equation}
In particular, for any $0\leq i\leq n$, $0\leq j\leq 2n$ and $x\in X$ the following statements hold:
\begin{itemize}
    \item[(1)] The sheaves $R^{n-i}q_*\Omega_{\Phi}^j$ are locally free with ranks given by the following formula $$\mathrm{rank}(R^{n-i}q_*\Omega_{\Phi}^j)=h^{n-i}(\Omega_{\Phi}^j|_{F_x})=\begin{cases}
        \displaystyle\binom{n}{j+i-n}h^{n-i,n-i}(E_x)+\displaystyle\sum\limits_{k=0}^j\displaystyle\binom{n}{k}h^{n-i,j-k}(X) & \textrm{ if $i\neq n$};\\
        \displaystyle\sum\limits_{k=0}^j\displaystyle\binom{n}{k}h^{0,j-k}(X) & \textrm{ if $i=n$.}
    \end{cases}$$ 
    \item[(2)] The natural maps $R^{n-i}q_*\Omega_{\Phi}^j\otimes\mathbb{C}(x)\to H^{n-i}(\Omega_{\Phi}^j|_{F_x})$ are isomorphisms.
\end{itemize}
\end{theorem}
(In the above, we use the notation $h^{i,j}(Z):=h^j(\Omega_Z^i)$ for any smooth variety $Z$.)

We remark that it is not hard to verify $h^{n-i}(\Omega_{\Phi}^j|_{F_x})=h^i(\Omega_{\Phi}^{2n-j}|_{F_x})$ using the formula above. However, this equality also follows directly as 
$$h^{n-i}(\Omega_{\Phi}^j|_{F_x})=h^i(\omega_{F_x}\otimes (\Omega_{\Phi}^{j})^*|_{F_x})=h^i(\Omega_{\Phi}^{2n-j}|_{F_x})$$ where the first equality comes from duality, and the second one follows from \eqref{kphi} and the isomorphism $(\Omega_{\Phi}^{j})^*\cong \Omega_{\Phi}^{2n-j}\otimes \omega_{\Phi}^*$.

\subsection{Filtrations \texorpdfstring{$(\Omega_{\Phi}^j|_{F_x}, {}_jG^{\bullet})$}{TEXT}, \texorpdfstring{$(\Omega_{E_{\Delta}}^j|_{E_x}, {}_jJ^{\bullet})$}{TEXT} and \texorpdfstring{$(\mathcal{K}_j,{}_jK^{\bullet})$}{TEXT}} In this section, we construct decreasing filtrations ${}_jG^{\bullet}$, ${}_jJ^{\bullet}$ and ${}_jK^{\bullet}$ on the sheaves $\Omega_{\Phi}^j|_{F_x}$, $\Omega_{E_{\Delta}}^j|_{E_x}$ and $\mathcal{K}_j$ respectively, where the last one is a certain sheaf  whose construction will also be explained. We study the properties of these filtrations which will be used in the proofs 
of various results. We first record the following useful lemma.

\begin{lemma}\label{compatibility}
    Suppose we have the following exact sequence of locally free sheaves:
\begin{equation}
    0\to\mathcal{A}\to\mathcal{B}\to\mathcal{C}\to 0.
\end{equation}
Then, for any $q$, there is a filtration
$${}_qF^0=\wedge^q\mathcal{B}\supseteq {}_qF^1\supseteq \cdots\supseteq {}_qF^q\supseteq {}_qF^{q+1}=0$$
satisfying the following properties for all $p$:
\begin{enumerate}
    \item $\mathrm{Gr}_{{}_qF}^p(\wedge^q\mathcal{B})\cong \wedge^p\mathcal{A}\otimes\wedge^{q-p}\mathcal{C}$,
    \item the natural maps induce a commutative diagram: 
    \begin{equation}\label{filtdiag}
        \begin{tikzcd}
            0\arrow[r] & \wedge^p\mathcal{A}\otimes{}_{q-p}F^1\arrow[r]\arrow[d] & \wedge^p\mathcal{A}\otimes \wedge^{q-p}\mathcal{B}\arrow[r]\arrow[d] & \wedge^p\mathcal{A}\otimes \wedge^{q-p}\mathcal{C}\arrow[r]\arrow[d, equal] & 0\\
            0\arrow[r] & {}_qF^{p+1}\arrow[r] & {}_qF^p\arrow[r] & \mathrm{Gr}^p_{{}_qF}(\wedge^q\mathcal{B})\cong\wedge^p\mathcal{A}\otimes \wedge^{q-p}\mathcal{C}\arrow[r] & 0
        \end{tikzcd}
    \end{equation}
    and the vertical maps are surjective.
\end{enumerate}
In particular, for a given $p$, if the map $H^i(\wedge^p\mathcal{A}\otimes \wedge^{q-p}\mathcal{B})\to H^i(\wedge^p\mathcal{A}\otimes \wedge^{q-p}\mathcal{C})$ is surjective, then so is the map $H^i({}_qF^p)\to H^i(\wedge^p\mathcal{A}\otimes\wedge^{q-p}\mathcal{C})$.
\end{lemma}
\begin{proof}
The filtration $(\wedge^q\mathcal{B},{}_qF^{\bullet})$ comes from \cite[Chapter II, Exercise 5.16]{Har} which satisfies (1). To see that it satisfies (2), recall the construction of this filtration: ${}_qF^p$ is the image of the composition $$\wedge^p\mathcal{A}\otimes\wedge^{q-p}\mathcal{B}\to \wedge^p\mathcal{B}\otimes\wedge^{q-p}\mathcal{B}\to\wedge^q\mathcal{B}$$ which gives us the surjective middle vertical map of the diagram \eqref{filtdiag}. It makes the right square of \eqref{filtdiag} commutative whence we get \eqref{filtdiag} with the surjective leftmost vertical map by snake lemma. This proves (2). To see the final assertion, consider the commutative diagram arising from \eqref{filtdiag}:
\begin{equation*}
    \begin{tikzcd}
        H^i(\wedge^p\mathcal{A}\otimes \wedge^{q-p}\mathcal{B})\arrow[r]\arrow[d] & H^i(\wedge^p\mathcal{A}\otimes \wedge^{q-p}\mathcal{C})\arrow[d, equal]\\
        H^i({}_qF^p)\arrow[r] & H^i(\wedge^p\mathcal{A}\otimes \wedge^{q-p}\mathcal{C})
    \end{tikzcd}
\end{equation*}
The top horizontal map is a surjection by assumption, whence so is the bottom horizontal map.
\end{proof}

Fix an integer $j\geq 1$ and $x\in X$. Let $(\Omega_{\Phi}^j|_{F_x}, {}_jG^{\bullet})$ and  $(\Omega_{E_{\Delta}}^j|_{E_x}, {}_jJ^{\bullet})$ be the filtrations coming from \lemmaref{compatibility} via \eqref{filt1} and \eqref{filt2} respectively. We now proceed to construct $(\mathcal{K}_j,{}_jK^{\bullet})$.

\smallskip

The compositions of the following surjections
\[
\begin{tikzcd}
\Omega_{\Phi}^1|_{F_x}\arrow[r, two heads] & \Omega_{\Phi}^1|_{E_x}\arrow[r, two heads] & \Omega_{E_{\Delta}}^1|_{E_x}
\end{tikzcd}
\]
and 
\[
\begin{tikzcd}
\Omega^1_{F_x}\arrow[r, two heads] & \Omega^1_{F_x}|_{E_x}\arrow[r, two heads] & \Omega^1_{E_x}
\end{tikzcd}
\]
make the two short exact sequences \eqref{filt1} and \eqref{filt2} fit into the following commutative diagram 
\begin{equation}\label{comp1}
    \begin{tikzcd}
    0\arrow[r] & \mathcal{O}_{F_x}^{\oplus n}\arrow[r]\arrow[d] & \Omega_{\Phi}^1|_{F_x}\arrow[r]\arrow[d, two heads] & \Omega_{F_x}^1\arrow[r]\arrow[d, two heads] & 0\\
    0\arrow[r] & \mathcal{O}_{E_x}^{\oplus n}\arrow[r] & \Omega^1_{E_{\Delta}}|_{E_x}\arrow[r] & \Omega_{E_x}^1\arrow[r] & 0
\end{tikzcd}
\end{equation}
where the middle and the right vertical surjective maps are described above.

\begin{claim}\label{claim1}
 The leftmost vertical arrow of \eqref{comp1} is just the coordinate-wise restriction, whence all three vertical maps of \eqref{comp1} are surjective.
\end{claim}
\begin{proof}
    Recall that the leftmost vertical map is identified with 
    \begin{equation}\label{coord}
        q^*\Omega_{X}^1|_{F_x}\to q_{\Delta}^*\Omega_{\Delta}^1|_{E_x}.
    \end{equation} 
    Let $$q_{F_x}:F_x\xrightarrow{b_x}\left\{x\right\}\times X\to\left\{x\right\}$$ be the composition (see middle vertical column of \eqref{diag2}). By \eqref{diag2}, we obtain $q^*\Omega_X^1|_{F_x}\cong q_{F_x}^*\mathcal{O}_{\left\{x\right\}}^{\oplus n}$.
    To this end, consider the commutative diagram:
    \begin{equation}\label{diag4}
        \begin{tikzcd}
        \left\{(x,x)\right\}\arrow[r, hook]\arrow[d, hook] & \Delta \arrow{dr}{\cong}[swap]{p_0} & \\
        \left\{x\right\}\times X\arrow[r] & \left\{x\right\}\arrow[r, hook] & X 
        \end{tikzcd}
    \end{equation}
It follows from \eqref{diag2}, \eqref{diag3} and \eqref{diag4} that $q_{\Delta}^*\Omega_{\Delta}^1|_{E_x}\cong q_{F_x}^*\mathcal{O}_{\left\{x\right\}}^{\oplus n}|_{E_x}$, whence \eqref{coord} can identified with the natural map $$q_{F_x}^*\mathcal{O}_{\left\{x\right\}}^{\oplus n}\to q_{F_x}^*\mathcal{O}_{\left\{x\right\}}^{\oplus n}|_{E_x}.$$ Thus the
conclusions follow.
\end{proof}

Similarly, for any $s$ we have the composition of surjections 
\begin{equation}\label{surj1}
\begin{tikzcd}
\Omega_{\Phi}^s|_{F_x}\arrow[r, two heads] & \Omega_{\Phi}^s|_{E_x}\arrow[r, two heads] & \Omega_{E_{\Delta}}^s|_{E_x},
\end{tikzcd}
\end{equation}
\begin{equation}\label{surj2}
    \begin{tikzcd}
\Omega^s_{F_x}\arrow[r, two heads] & \Omega^s_{F_x}|_{E_x}\arrow[r, two heads] & \Omega^s_{E_x}
\end{tikzcd}
\end{equation}
which are evidently induced by \eqref{comp1}. Now, \eqref{surj2} induces the following surjections via coordinate-wise maps for any $k$:
\begin{equation}\label{surj3}
    \begin{tikzcd}
        \text{Gr}_{{}_jG}^k(\Omega_{\Phi}^j|_{F_x})\cong \wedge^k(\mathcal{O}_{F_x}^{\oplus n})\otimes \Omega_{F_x}^{j-k}\arrow[r, two heads] & \text{Gr}_{{}_jJ}^k(\Omega_{E_{\Delta}}^j|_{E_x})\cong \wedge^k(\mathcal{O}_{E_x}^{\oplus n})\otimes \Omega_{E_x}^{j-k}.
    \end{tikzcd}
\end{equation}
Observe that the surjections $\Omega_{\Phi}^s|_{F_x}\to\Omega_{F_x}^s$ and $\Omega^s_{E_{\Delta}}|_{E_x}\to \Omega_{E_x}^s$ arising from \eqref{filt1} and \eqref{filt2} respectively are compatible with \eqref{surj1} and \eqref{surj2} (see \eqref{comp1}), whence using \eqref{surj3}, we inductively define the maps 
\begin{equation}\label{surj4}
    {}_jG^s\to {}_jJ^s
\end{equation}
which make the following diagram commutative (thanks to \claimref{claim1}):
\begin{equation}\label{kfilt}
\begin{tikzcd}
    0\arrow[r] & {}_jG^{k+1}\arrow[r]\arrow[d] & {}_jG^k\arrow[r]\arrow[d] & \text{Gr}_{{}_jG}^k(\Omega_{\Phi}^j|_{F_x})\arrow[r]\arrow[d, two heads] & 0\\
    0\arrow[r] & {}_jJ^{k+1}\arrow[r] & {}_jJ^k\arrow[r] & \text{Gr}_{{}_jJ}^k(\Omega_{E_{\Delta}}^j|_{E_x}) \arrow[r] & 0
\end{tikzcd}
\end{equation}
Recall that the right vertical map \eqref{surj3} is surjective. 

\begin{claim}\label{claim2}
    The maps ${}_jG^s\to {}_jJ^s$ in \eqref{surj4} are surjective for all $s$. In particular all three vertical maps in \eqref{kfilt} are surjective.
\end{claim}
\begin{proof}
    We use the commutative diagram 
\[
\begin{tikzcd}
    \wedge^s\mathcal{O}_{F_x}^{\oplus n}\otimes \Omega^{j-s}_{\Phi}|_{F_x}\arrow[r, two heads]\arrow[d, two heads] & \wedge^s\mathcal{O}_{E_x}^{\oplus n}\otimes \Omega^{j-s}_{E_{\Delta}}|_{E_x}\arrow[d, two heads]\\
    {}_jG^s\arrow[r] & {}_jJ^s
\end{tikzcd}
\]
where the top horizontal map is evidently surjective, and the vertical maps are surjective by construction (see \lemmaref{compatibility} (2)). Consequently the bottom horizontal map is also surjective.
\end{proof}

To this end, we define $$\mathcal{K}_j:=\text{Ker}\left(\Omega_{\Phi}^j|_{F_x}\to \Omega^j_{E_{\Delta}}|_{E_x}\right)$$ arising from \eqref{surj1}, and $${}_jK^k:=\textrm{Ker}\left({}_jG^k\to {}_jJ^k\right).$$ This way we obtain a filtration $(\mathcal{K}_j,{}_jK^{\bullet})$. Moreover by \claimref{claim2} and snake lemma, we obtain  \begin{equation}\label{aux}
\text{Gr}_{{}_jK}^k(\mathcal{K}_j)\cong \textrm{Ker}\left(\text{Gr}_{{}_jG}^k(\Omega_{\Phi}^j|_{F_x})\to  \text{Gr}_{{}_jJ}^k(\Omega_{E_{\Delta}}^j|_{E_x})\right)\cong \wedge^k\mathcal{O}_{F_x}^{\oplus n}\otimes\Omega_{F_x}^{j-k}(\log E_x)(-E_x)
\end{equation}
where the last isomorphism is a consequence of the exact sequence 
\begin{equation}\label{aux4}
    0\to \Omega_{F_x}^{j-k}(\log E_x)(-E_x)\to \Omega_{F_x}^{j-k}\to \Omega_{E_x}^{j-k}\to 0.
\end{equation}

\smallskip

We summarize the discussion above by the following commutative diagram with exact rows and columns:
\begin{equation}\label{mainkfilt}
\begin{tikzcd}
    & 0\arrow[d] & 0\arrow[d] & 0 \arrow[d] &\\
    0\arrow[r] & {}_jK^{k+1}\arrow[r]\arrow[d] & {}_jK^k\arrow[d]\arrow[r] & \text{Gr}_{{}_jK}^k(\mathcal{K}_j)\arrow[r]\arrow[d] & 0\\
    0\arrow[r] & {}_jG^{k+1}\arrow[r]\arrow[d] & 
    {}_jG^k\arrow[r]\arrow[d] & \text{Gr}_{{}_jG}^k(\Omega_{\Phi}^j|_{F_x})\arrow[r]\arrow[d] & 0\\
    0\arrow[r] & {}_jJ^{k+1}\arrow[r]\arrow[d] & {}_jJ^k\arrow[r]\arrow[d] & \text{Gr}_{{}_jJ}^k(\Omega_{E_{\Delta}}^j|_{E_x}) \arrow[r]\arrow[d] & 0\\
    & 0 & 0 & 0 &
\end{tikzcd}
\end{equation}

In what follows, we will frequently use the fact 
$$H^{s,t}(E_x)=\begin{cases} 
      \mathbb{C} & \textrm{if }\, s=t\leq n-1; \\
      0 & \textrm{ otherwise }
   \end{cases}$$
without any further reference, which holds since $E_x\cong\mathbb{P}^{n-1}$.

\subsection{Key properties of \texorpdfstring{$R^iq_*\Omega_{\Phi}^j$}{TEXT}} We now proceed to the proof of \theoremref{mainfiltthm}. First we need some preparations:

\begin{lemma}\label{filtlemma}
    Assume $j\geq 1$. The following statements hold:
    \begin{enumerate}
        \item $H^i({}_jJ^{k+1})\cong H^i({}_jJ^{k})$ for all $k\leq j-i-1$ via the map arising from the third row of \eqref{mainkfilt},
        \item $H^i({}_jJ^{k})=0$ for all $k\geq j-i+1$,
        \item $H^i({}_jJ^{j-i})\cong H^i(\mathrm{Gr}_{{}_jJ}^{j-i}(\Omega_{E_{\Delta}}^j|_{E_x}))$ via the map arising from the third row of \eqref{mainkfilt}.
        \item $H^i({}_jJ^k)\to H^i(\mathrm{Gr}_{{}_jJ}^k(\Omega_{E_{\Delta}}^j|_{E_x}))$ is a surjection.
        \item $h^i({}_jJ^{k})=\binom{n}{j-i}h^{i,i}(E_x)$ for all $k\leq j-i$. In particular, $h^i({}_jJ^{0})=\binom{n}{j-i}h^{i,i}(E_x)$.
    \end{enumerate}
\end{lemma}
\begin{proof}
We first observe that 
\begin{equation}\label{small}
    H^i(\mathrm{Gr}_{{}_jJ}^k(\Omega_{E_{\Delta}}^j|_{E_x}))=0\,\textrm{ for all }\, k\neq j-i. 
\end{equation}
Consequently (1) follows by passing to the cohomology of the third row of \eqref{mainkfilt}. 

To see (2), first note that for trivial reason, we have
\begin{equation*}
    H^i({}_jJ^{j+1})=0.
\end{equation*}
Let $k=j-i+r$ with $r\geq 1$. The assertion follows easily by decreasing induction on $r$ from \eqref{small}, with $r=i+1$ being the base case, for which the assertion has been established above.

Now (3) is an immediate consequence of (2). Assertion (4) follows from \eqref{small} and (3). Finally (5) is a consequence of (1) and (3) (we also need (2) for the last assertion).
\end{proof}

We make use of the first two rows of \eqref{diag2} in the proof of the following

\begin{lemma}\label{filtlemma2}
    Assume $j\geq 1$. Then the map $$H^i({}_jG^k)\to H^i(\mathrm{Gr}_{{}_jG}^k(\Omega_{\Phi}^j|_{F_x}))$$ arsing from the second row of \eqref{mainkfilt} is surjective.
\end{lemma}
\begin{proof}
Thanks to \lemmaref{compatibility}, we are required to show the surjection of $$H^i(\wedge^k(\mathcal{O}_{F_x}^{\oplus n})\otimes \Omega_{\Phi}^{j-k}|_{F_x})\to H^i(\wedge^k(\mathcal{O}_{F_x}^{\oplus n})\otimes \Omega^{j-k}_{F_x}).$$ This map is induced coordinate-wise by the map $H^i(\Omega_{\Phi}^{j-k}|_{F_x})\to H^i(\Omega_{F_x}^{j-k})$ whence we only need to prove the following 
\begin{claim}\label{random}
    The natural map $H^i(\Omega_{\Phi}^{j-k}|_{F_x})\to H^i(\Omega_{F_x}^{j-k})$ is a surjection.
\end{claim}
\begin{proof}
It is enough to show that the composed map 
\begin{equation}\label{needfilt}
    H^i(\Omega_{\Phi}^{j-k})\to H^i(\Omega_{\Phi}^{j-k}|_{F_x})\to H^i(\Omega_{F_x}^{j-k})
\end{equation}
is a surjection. Now there are two cases possible that we analyze below. In what follows, we use the fact that 
\begin{equation}\label{hpq}
    H^i(\Omega_{F_x}^s)=\begin{cases} 
      H^i(\Omega_X^s) & \textrm{if }\, i\neq s,\, \textrm{ or }\, (i,s)\in\left\{(0,0),(n,n)\right\}; \\
      H^i(\Omega_X^s)\oplus\mathbb{C}[c_1(E_{x})^s] & \textrm{ if } 1\leq i=s\leq n-1.
   \end{cases}
\end{equation}

\noindent{\bf Case $1$: $k\neq j-i$ or $i=j-k\in \left\{0,n\right\}$.} The blow-up exact sequences of $b_{\Delta}$ and $b_x$ gives the commutative diagram:
\[
\begin{tikzcd}
    H^i(\Omega_{\Phi}^{j-k})\arrow[r] & H^i(\Omega_{F_x}^{j-k})\\
    H^i(\Omega^{j-k}_{X\times X})\arrow[u]\arrow[r, two heads] & H^i(\Omega_X^{j-k})\arrow[u, swap, "\cong"]
\end{tikzcd}
\]
(see \eqref{diag2}) where we identify $\Omega^1_{\left\{x\right\}\times X}$ with $\Omega_X^1$. The bottom horizontal map is clearly surjective as it is the projection onto a direct summand, and the right vertical map is also surjective. Consequently the composed map \eqref{needfilt} is also surjective.

\smallskip

\noindent\textbf{Case $2$: $1\leq i=j-k\leq n-1$.} In this case, we again work with the commutative diagram
\[
\begin{tikzcd}
    H^i(\Omega_{\Phi}^{i})\arrow[r] & H^i(\Omega_{F_x}^{i})\cong H^i(\Omega_X^i)\oplus\mathbb{C}[c_1(E_{x})^i]\\
    H^i(\Omega^{i}_{X\times X})\arrow[u]\arrow[r, two heads] & H^i(\Omega_X^{i})\arrow[u]
\end{tikzcd}
\]
and we immediately see that it is enough to show that $c_1(E_{x})^i\in H^{i,i}({E_x})$ has a preimage in $H^{i,i}(\Phi)$. Recall that $E_x=E_{\Delta}\cap F_x$ whence $c_1(E_{\Delta})^i\in H^{i,i}(\Phi)$ is the required preimage. 
\end{proof}
The proof of the lemma follows immediately from the above claim.
\end{proof}

We record the following obvious 

\begin{remark}\label{filtrmk}
    The map $$H^i(\text{Gr}_{{}_jG}^k(\Omega_{\Phi}^j|_{F_x}))\to H^i(\text{Gr}_{{}_jJ}^k(\Omega_{E_{\Delta}}^j|_{E_x}))$$ arising from the right column of \eqref{mainkfilt} is surjective for all $i,j,k$. Indeed, this is immediate from the surjections $H^{i,i}(F_x)\to H^{i,i}(E_x)$ for all $i$.
\end{remark}

\begin{lemma}\label{filtlemma3}
Assume $j\geq 1$. Then the map $${}_ja_k: H^i({}_jG^k)\to H^i({}_jJ^k)$$ arising from the middle column of \eqref{mainkfilt} is surjective. 
\end{lemma}
\begin{proof}
Thanks to \lemmaref{filtlemma} (2), we assume $k\leq j-i$ and we work with the commutative diagram:
\footnotesize
\begin{equation*}
\begin{tikzcd}
 H^i(\text{Gr}_{{}_jG}^{j-i}(\Omega_{\Phi}^j|_{F_x}))\arrow[d, two heads] & H^i({}_jG^{j-i}\arrow[l, two heads]\arrow[d, "{}_ja_{j-i}"])\arrow[r] & \cdots \arrow[r]\arrow[d, dotted] & H^i({}_jG^k)\arrow[r]\arrow[d, "{}_ja_k"] & \cdots\arrow[r]\arrow[d, dotted] & H^i({}_jG^1)\arrow[r]\arrow[d, "{}_ja_1"] & H^i(\Omega_{\Phi}^j|_{F_x})\arrow[d, "{}_ja_0"]\\
 H^i(\text{Gr}_{{}_jJ}^{j-i}(\Omega_{E_{\Delta}}^j|_{E_x})) & H^i({}_jJ^{j-i})\arrow[l, swap, "\cong"]\arrow[r, "\cong"] & \cdots\arrow[r, "\cong"] & H^i({}_jJ^k)\arrow[r, "\cong"] & \cdots\arrow[r, "\cong"] & H^i({}_jJ^1)\arrow[r, "\cong"] & H^i(\Omega_{E_{\Delta}}^j|_{E_x})
\end{tikzcd}
\end{equation*}
\normalsize
where the bottom horizontal maps are isomorphisms by \lemmaref{filtlemma} (1) and (3), the leftmost top horizontal map is a surjection by \lemmaref{filtlemma2} and the leftmost vertical map is a surjection by \remarkref{filtrmk}. Consequently ${}_ja_{j-i}$ is surjective, whence ${}_ja_k$ is surjective as well.
\end{proof}

\begin{proof}[Proof of \theoremref{mainfiltthm}]
We first prove \eqref{recursion}. The assertion is obvious when $j=0$, so we assume $j\geq 1$. 

Using \lemmaref{filtlemma} (4), \lemmaref{filtlemma2}, \remarkref{filtrmk}, \lemmaref{filtlemma3} and snake lemma, we obtain from \eqref{mainkfilt} the following commutative diagram with exact rows and columns for any $j\geq 1$:
\begin{equation}\label{aux1}
\begin{tikzcd}
    & 0\arrow[d] & 0\arrow[d] & 0 \arrow[d] &\\
    0\arrow[r] & H^i({}_jK^{k+1})\arrow[r]\arrow[d] & H^i({}_jK^k)\arrow[d]\arrow[r] & H^i(\text{Gr}_{{}_jK}^k(\mathcal{K}_j))\arrow[r]\arrow[d] & 0\\
    0\arrow[r] & H^i({}_jG^{k+1})\arrow[r]\arrow[d] & 
    H^i({}_jG^k)\arrow[r]\arrow[d] & H^i(\text{Gr}_{{}_jG}^k(\Omega_{\Phi}^j|_{F_x}))\arrow[r]\arrow[d] & 0\\
    0\arrow[r] & H^i({}_jJ^{k+1})\arrow[r]\arrow[d] & H^i({}_jJ^k)\arrow[r]\arrow[d] & H^i(\text{Gr}_{{}_jJ}^k(\Omega_{E_{\Delta}}^j|_{E_x})) \arrow[r]\arrow[d] & 0\\
    & 0 & 0 & 0 &
\end{tikzcd}
\end{equation}
Also recall from \eqref{aux} that 
\begin{equation}\label{auxx}
    h^i(\text{Gr}_{{}_jK}^k(\mathcal{K}_j))=\binom{n}{k}h^i(\Omega_{F_x}^{j-k}(\log E_x)(-E_x)).
\end{equation}
We claim that 
\begin{equation}\label{aux2}
    h^i(\Omega_{F_x}^{j-k}(\log E_x)(-E_x))=h^{i,j-k}(F_x)-h^{i,j-k}(E_x).
\end{equation}
In view of the exact sequence \eqref{aux4}, it is enough to show that 
\begin{equation}\label{aux3}
    H^i(\Omega_{F_x}^{j-k})\to H^i(\Omega_{E_x}^{j-k})\,\textrm{ arising from \eqref{aux4} is surjective.}
\end{equation}
This is immediate if $j-k\neq i$; when $j-k=i$, the assertion follows since in this case the image of $c_1(E_x)^i\in H^{i,i}(F_x)$ is non-vanishing.

\smallskip

Now \eqref{recursion} follows immediately from \lemmaref{filtlemma} (5), \eqref{aux1}, \eqref{auxx} and\eqref{aux2}. Also, the formula for $h^{n-i}(\Omega_{\Phi}^j|_{F_x})$ follows from \eqref{recursion} and \eqref{hpq} upon simplification. Finally, the remaining assertions in (1) and (2) follows from Grauert's theorem (see \cite[Chapter III, Corollary 12.9]{Har}). 
The proof is now complete.
\end{proof}

\section{Local cohomology of the secant varieties}\label{secfinal} We resume our notation and hypothesis. To emphasize, throughout we assume $L$ is 3-very ample and $\Sigma\neq \mathbb{P}^N:=\mathbb{P}(H^0(L))$. 

\smallskip By \cite[Proposition 3.3]{Ste} (see also \cite[Sect. \S2.1]{MPOW}), we have a distinguished triangle:
\begin{equation}\label{ste}
    {\bf R}t_*\Omega^s_{\mathbb{P}(\mathcal{E}_L)}(\log\Phi)(-\Phi)\to \underline{\Omega}^s_{\Sigma}\to \underline{\Omega}^s_{X}\xrightarrow{+1}
\end{equation}
Dualizing the above, we obtain the distinguished triangle that will be used later on:
\begin{equation}\label{dualdist}
    {\bf D}_{\Sigma}(\underline{\Omega}_X^{s})\to {\bf D}_{\Sigma}(\underline{\Omega}_{\Sigma}^{s})\to {\bf R}t_*\Omega_{\mathbb{P}(\mathcal{E}_L)}^{2n+1-s}(\log\Phi)\xrightarrow{+1}.
\end{equation}
We will use, often without stating, the the complexes $\underline{\Omega}^s_{\Sigma}$ and ${\bf D}_{\Sigma}(\underline{\Omega}_{\Sigma}^{s})$ are supported in non-negative degrees. We also recall the following

\begin{lemma}[{\cite[Remark 6.1, Lemma 6.3]{ORS}}]\label{h0} The following assertions hold:
\begin{enumerate}
    \item We have 
    \begin{equation}\label{hid1}
{\bf D}_{\Sigma}(\underline{\Omega}_X^{2n+1-k})\cong {\bf D}_X(\underline{\Omega}_X^{2n+1-k})[-n-1]\cong \begin{cases} 
      \Omega_X^{k-n-1}[-n-1] & \textrm{if }\, k\geq n+1; \\
      0 & \textrm{otherwise}.
   \end{cases}
\end{equation}  
In particular, 
\begin{equation}\label{hid}
    \mathcal{H}^i({\bf D}_{\Sigma}(\underline{\Omega}_{X}^{2n+1-k}))\cong \begin{cases} 
      \Omega_X^{k-n-1} & \textrm{if }\, i=n+1\, \textrm{ and }\, k\geq n+1; \\
      0 & \textrm{otherwise}.
   \end{cases}
\end{equation}
    \item If one of the following conditions hold:
    \begin{itemize}
        \item[(a)] $0\leq i\leq n-1$, $0\leq k\leq 2n+1$; or
        \item[(b)] $i\geq 0$, $0\leq k\leq n$,
    \end{itemize}
    then $$\mathcal{H}^i({\bf D}_{\Sigma}(\underline{\Omega}_{\Sigma}^{2n+1-k}))\cong R^it_*\Omega_{\mathbb{P}(\mathcal{E}_L)}^k(\log\Phi).$$
    \item $\mathcal{H}^i({\bf D}_{\Sigma}(\underline{\Omega}_{\Sigma}^{2n+1-k}))=0$ for all $i\geq n+2$, $0\leq k\leq 2n+1$.
\end{enumerate}
\end{lemma}
We aim to compute $\textrm{lcd}(\mathbb{P}^N,\Sigma)$ and the generation level of the Hodge filtration on $\mathcal{H}_{\Sigma}^{\textrm{lcd}(\mathbb{P}^N,\Sigma)}(\mathcal{O}_{\mathbb{P}^N})$.

\smallskip

Let us first state some standard facts that will often be used without any further reference:

(1) A sheaf $\mathcal{F}$ on a smooth variety $Z$ is locally free if and only if $\mathcal{E}\textrm{\textit{xt}}^i(\mathcal{F},\omega_Z)=0$ for all $i>0$.

(2) If $\varphi:\mathcal{F}\to\mathcal{F''}$ is a surjective morphism between locally free sheaves on a smooth variety $Z$, then $\textrm{Ker}(\varphi)$ is locally free. Indeed, setting $\mathcal{F}':=\textrm{Ker}(\varphi)$, we obtain the exact sequence $$0\to\mathcal{F}'\to \mathcal{F}\xrightarrow{\varphi}\mathcal{F''}\to 0.$$ Using (1), we conclude $\mathcal{E}\textrm{\textit{xt}}^i(\mathcal{F}',\omega_Z)=0$ for $i\geq 1$, whence $\mathcal{F}'$ is locally free. 
In particular, a surjective morphism between locally free sheaves of equal ranks on a smooth variety is an isomorphism.

\subsection{Proof of \theoremref{mainlcd}} From now on, we work with the following notation: $\sigma: \mathbb{P}\to\mathbb{P}^N$ is an embedded log resolution of $(\mathbb{P}^N,\Sigma)$ which we assume to be isomorphism over $\mathbb{P}^N\backslash\Sigma$. Further, we set $E:=\sigma^{-1}(\Sigma)_{\textrm{red}}$ which is a simple normal crossing divisor on $\mathbb{P}$. 

\begin{proposition}\label{freelcd} Assume $L$ satisfies $(Q_0)$-property. Then 
$$R^i\sigma_*\omega_{\mathbb{P}}(E)=R^i\sigma_*\omega_E=0\,\textrm{ for all }\,i\geq q_X-\nu(X)-2.$$ 
\end{proposition}


\begin{proof} Observe that using the exact sequence $$0\to \omega_{\mathbb{P}}\to \omega_{\mathbb{P}}(E)\to \omega_{E}\to 0$$
and Grauert-Riemenschneider vanishing (\cite[Theorem 4.3.9]{Laz}) which says that $R^k\sigma_*\omega_{\mathbb{P}}=0$ for all $k\geq 1$, we obtain 
\begin{equation*}
    R^i\sigma_*\omega_{\mathbb{P}}(E)\cong R^i\sigma_*\omega_E\,\textrm{ for all }\, i\geq 1.
\end{equation*}
We aim to prove $R^i\sigma_*\omega_E=0$ for $i\geq q_X-\nu(X)-2$. It follows from \cite[Corollary B]{MP} that the natural map $$\mathcal{H}^i({\bf R}\sigma_*\omega_E^{\bullet})\to\mathcal{H}^i(\omega_{\Sigma}^{\bullet})$$ is an injection for all $i$. Using the isomorphisms $$\omega_E^{\bullet}\cong\omega_E[N-1]\,\textrm{ and }\, \omega_{\Sigma}^{\bullet}\cong{\bf R}\mathcal{H}\textrm{\textit{om}}_{\mathcal{O}_{\mathbb{P}^N}}(\mathcal{O}_{\Sigma},\omega_{\mathbb{P}^N}[N])$$ we find that $R^i\sigma_*\omega_E\to \mathcal{E}\text{\textit{xt}}_{\mathcal{O}_{\mathbb{P}^N}}^{i+1}(\mathcal{O}_{\Sigma},\omega_{\mathbb{P}^N})$ is an injection for all $i$. It is well known that $$\mathcal{E}\text{\textit{xt}}_{\mathcal{O}_{\mathbb{P}^N}}^{i+1}(\mathcal{O}_{\Sigma},\omega_{\mathbb{P}^N})=0\,\textrm{ for all }\, i\geq \textrm{pd}(\mathcal{O}_{\Sigma}).$$ On the other hand, recall that we have $\textrm{depth}(\mathcal{O}_{\Sigma})=n+2+\nu(X)$, whence by Auslander-Buchsbaum formula, we deduce that $$\textrm{pd}(\mathcal{O}_{\Sigma})=q_X-\nu(X)-2.$$ Thus $R^i\sigma_*\omega_E=0$ for all $i\geq q_X-\nu(X)-2$. 
\end{proof}

We need some more preparations for the proof of \theoremref{mainlcd}:

\begin{lemma}\label{lcdlemma1}
    For all $i\leq q_X-1$ and $0\leq j\leq N$, we have the isomorphism
    $$R^{q_X-i}\sigma_*\Omega_{\mathbb{P}}^{N-j}(\log E)\cong \mathcal{H}^{n+2-i}({\bf D}_{\Sigma}(\underline{\Omega}_{\Sigma}^j)).$$
\end{lemma}

\begin{proof} Applying ${\bf D}_{\mathbb{P}^N}(-)$ on the following distinguished triangle (recall \cite[Proposition 3.3]{Ste})
$${\bf R}\sigma_*\Omega_{\mathbb{P}}^j(\log E)(-E)\to \underline{\Omega}_{\mathbb{P}^N}^j\to\underline{\Omega}_{\Sigma}^j\xrightarrow{+1},$$
we obtain the distinguished triangle
\begin{equation*}
    {\bf D}_{\Sigma}(\underline{\Omega}_{\Sigma}^j)[-q_{\Sigma}]\to \Omega_{\mathbb{P}^N}^{N-j}\to {\bf R}\sigma_*\Omega_{\mathbb{P}}^{N-j}(\log E)\xrightarrow{+1}.
\end{equation*}
The assertion follows by taking the cohomology of the above.
\end{proof}

\begin{lemma}\label{lcdlemma2}
The following statements hold for all $0\leq j\leq N$:
\begin{enumerate}
    \item For all $3\leq i\leq q_X-1$, we have the isomorphism $$R^{q_X-i}\sigma_*\Omega_{\mathbb{P}}^{N-j}(\log E)\cong R^{n+2-i}t_*\Omega_{\mathbb{P}(\mathcal{E}_L)}^{2n+1-j}(\log\Phi).$$
    \item We have an exact sequence\footnote{Here we use the convention that for a variety $Z$, $\Omega_Z^{-k}=0$ if $k\geq 1$} $$0\to R^{q_X-2}\sigma_*\Omega_{\mathbb{P}}^{N-j}(\log E)\to R^{n}t_*\Omega_{\mathbb{P}(\mathcal{E}_L)}^{2n+1-j}(\log\Phi)\xrightarrow{f_j} \Omega_X^{n-j}\to R^{q_X-1}\sigma_*\Omega_{\mathbb{P}}^{N-j}(\log E)\to 0.$$
    \item $R^{q_X-i}\sigma_*\Omega_{\mathbb{P}}^{N-j}(\log E)=0$ for all $i\leq 0$.
\end{enumerate}
\end{lemma}
\begin{proof} We only give the proof when $0\leq j\leq 2n+1$; the case when $j\geq 2n+2$ is similar and straightforward. We apply \lemmaref{lcdlemma1}. The assertion (1) is an immediate consequence of  \lemmaref{h0} (2) (when $n+2-i<0$, the assertion is trivial since both sides are zero). To prove (2), we take the cohomology of the distinguished triangle \eqref{dualdist} corresponding to $s=j$. Recall that $$R^{n+1}t_*\Omega_{\mathbb{P}(\mathcal{E}_L)}^{2n+1-j}(\log \Phi)=0$$
whence the conclusion follows from \eqref{hid}. Lastly, (3) is a consequence of \lemmaref{h0} (3).
\end{proof}

\begin{lemma}\label{lcdlemma3}
The following statements hold for all $0\leq j\leq N$:
\begin{enumerate}
    \item ${\bf D}_{\Sigma}({\bf R} q_*\Omega_{\Phi}^j)\cong {\bf R}q_*\Omega_{\Phi}^{2n-j}[-1]$.
    \item The map $f_j$ in \lemmaref{lcdlemma2} (2) is the composition $h_j\circ g_j$ where $g_j$ and $h_j$ are as follows:
    \begin{itemize}
        \item $g_j:R^nt_*\Omega^{2n+1-j}_{\mathbb{P}(\mathcal{E}_L)}(\log\Phi)\to R^nq_*\Omega_{\Phi}^{2n-j}$ arising as $$\mathcal{H}^n({\bf D}_{\Sigma}({\bf R}t_*\Omega_{\mathbb{P}(\mathcal{E}_L)}^j(\log \Phi)(-\Phi)))\to \mathcal{H}^{n+1}({\bf D}_{\Sigma}({\bf R} q_*\Omega_{\Phi}^j))$$ using (1), from the distinguished triangle
        $${\bf R}t_*\Omega_{\mathbb{P}(\mathcal{E}_L)}^j(\log \Phi)(-\Phi)\to {\bf R}t_*\Omega_{\mathbb{P}(\mathcal{E}_L)}^j\to {\bf R}q_*\Omega_{\Phi}^j\xrightarrow{+1}.$$
        \item $h_j: R^nq_*\Omega_{\Phi}^{2n-j}\to \Omega_X^{n-j}$ arising as $\mathcal{H}^{n+1}({\bf D}_{\Sigma}({\bf R} q_*\Omega_{\Phi}^j))\to \mathcal{H}^{n+1}({\bf D}_{\Sigma}(\underline{\Omega}_X^j))$ using (1) and \eqref{hid}, from the natural map $\Omega_X^j\to {\bf R}q_*\Omega_{\Phi}^j$. 
    \end{itemize}
    \item The map $g_j:R^nt_*\Omega^{2n+1-j}_{\mathbb{P}(\mathcal{E}_L)}(\log\Phi)\to R^nq_*\Omega_{\Phi}^{2n-j}$ is surjective.
\end{enumerate}
\end{lemma}
\begin{proof} We observe that $${\bf D}_{\Sigma}({\bf R} q_*\Omega_{\Phi}^j)\cong {\bf D}_X({\bf R} q_*\Omega_{\Phi}^j)[-n-1]\cong {\bf R}q_*{\bf D}_{\Phi}(\Omega_{\Phi}^j)[n][-n-1]\cong {\bf R}q_*\Omega_{\Phi}^{2n-j}[-1]$$
whence (1) follows. To see (2), we dualize the commutative diagram (the top row is \eqref{ste} with $s=j$, the bottom row is obtained from \eqref{prev1})
\begin{equation*}
\begin{tikzcd}
    {\bf R}t_*\Omega_{\mathbb{P}(\mathcal{E}_L)}^j(\log \Phi)(-\Phi)\arrow[r]\arrow[d, "\cong"] &  \underline{\Omega}_{\Sigma}^j\arrow[r]\arrow[d] & \underline{\Omega}_{X}^j\arrow[d]\arrow[r, "+1"] & \textrm{ }\\
    {\bf R}t_*\Omega_{\mathbb{P}(\mathcal{E}_L)}^j(\log \Phi)(-\Phi)\arrow[r] & {\bf R}t_*\Omega_{\mathbb{P}(\mathcal{E}_L)}^j\arrow[r] & {\bf R}q_*\Omega_{\Phi}^j\arrow[r, "+1"] & \textrm{}
\end{tikzcd}
\end{equation*}
and use \eqref{hid1} to obtain:
\begin{equation*}
\begin{tikzcd}
    {\bf D}_{\Sigma}({\bf R}q_*\Omega_{\Phi}^j)\arrow[r]\arrow[d] & {\bf R}t_*\Omega_{\mathbb{P}(\mathcal{E}_L)}^{2n+1-j}\arrow[r]\arrow[d] & {\bf R}t_*\Omega_{\mathbb{P}(\mathcal{E}_L)}^{2n+1-j}(\log \Phi)\arrow[r, "+1"]\arrow[d, "\cong"] & \textrm{ }\\
    \Omega_X^{n-j}[-n-1]\arrow[r] & {\bf D}_{\Sigma}(\underline{\Omega}_{\Sigma}^j)\arrow[r] & {\bf R}t_*\Omega_{\mathbb{P}(\mathcal{E}_L)}^{2n+1-j}(\log \Phi)\arrow[r, "+1"] & \textrm{ }
\end{tikzcd}
\end{equation*}
whence the assertion follows. Finally, (3) is a consequence of the description of $g_j$; it is surjective since $$\mathcal{H}^{n+1}({\bf D}_{\Sigma}({\bf R}t_*\Omega_{\mathbb{P}(\mathcal{E}_L)}^j))\cong \mathcal{H}^{n+1}({\bf R}t_*\Omega_{\mathbb{P}(\mathcal{E}_L)}^{2n+1-j})\cong R^{n+1}t_*\Omega_{\mathbb{P}(\mathcal{E}_L)}^{2n+1-j}=0.$$ The proof is now complete.
\end{proof}

In what follows, we set  $\mathcal{C}_j^{\bullet}$ to be the cone of $\Omega_X^j\to {\bf R}q_*\Omega_{\Phi}^j$, i.e., we have the distinguished triangle: 
\begin{equation}\label{cone-exact}
    \mathcal{C}_j^{\bullet}\to \Omega_X^j\to {\bf R}q_*\Omega_{\Phi}^j\xrightarrow{+1}.
\end{equation}

\begin{lemma}\label{lcdlemma4}
    If $\mathcal{E}\textrm{xt}^p_{\mathcal{O}_X}(\mathcal{H}^p(\mathcal{C}_j^{\bullet}),\omega_X)=0$ for all $p\geq 0$, then the following conclusions hold:
    \begin{enumerate}
        \item $h_j: R^nq_*\Omega_{\Phi}^{2n-j}\to \Omega_X^{n-j}$ described in \lemmaref{lcdlemma3} (2) is surjective.
        \item $R^{q_X-1}\sigma_*\Omega_{\mathbb{P}}^{N-j}(\log E)=0$. 
    \end{enumerate}
\end{lemma}
\begin{proof} By \lemmaref{lcdlemma2} (2), the assertion (2) holds if and only if $f_j$ is surjective. From \lemmaref{lcdlemma3} (2) and (3) respectively, we see that $f_j=h_j\circ g_j$ and $g_j$ is surjective. 
Thus, (2) is a consequence of (1), i.e. it is enough to show that $h_j$ is surjective under our hypothesis. 

Dualizing \eqref{cone-exact} and passing to its cohomology, we see that $h_j$ is surjective if $\mathcal{H}^{n+1}({\bf D}_{\Sigma}(\mathcal{C}_j^{\bullet}))=0$. Note that
$$\mathcal{H}^{n+1}({\bf D}_{\Sigma}(\mathcal{C}_j^{\bullet}))\cong \mathcal{H}^{n+1}({\bf D}_{X}(\mathcal{C}_j^{\bullet})[-n-1])\cong \mathcal{H}^0({\bf D}_{X}(\mathcal{C}_j^{\bullet})),$$ whence it is enough to show that
\begin{equation}\label{h0need}
    \mathcal{H}^0({\bf D}_{X}(\mathcal{C}_j^{\bullet}))=0.
\end{equation}
Observe that $\mathcal{H}^0({\bf D}_{X}(\mathcal{C}_j^{\bullet}))\cong \mathcal{H}\textit{om}_{\mathcal{O}_X}(\mathcal{C}_j^{\bullet},\omega_X)$. To this end, we use the spectral sequence 
$$E_2^{p,q}=\mathcal{E}\textit{xt}^p_{\mathcal{O}_X}(\mathcal{H}^{-q}(\mathcal{C}_j^{\bullet}),\omega_X)\implies \mathcal{E}\textit{xt}^{p+q}_{\mathcal{O}_X}(\mathcal{C}_j^{\bullet},\omega_X).$$
Now, $p+q=0\implies q=-p$, whence to prove \eqref{h0need} it is enough to show that 
$\mathcal{E}\textit{xt}^p_{\mathcal{O}_X}(\mathcal{H}^{p}(C^{\bullet}),\omega_X)=0$ for all $p$. Clearly this holds by our assumption since it follows from \eqref{cone-exact} that $\mathcal{H}^p(\mathcal{C}_j^{\bullet})=0$ for all $p<0$.
\end{proof}

\begin{proposition}\label{lcdcor} The following statements hold for all $j\geq 1$:
\begin{enumerate}
    \item $\mathcal{H}^0(\mathcal{C}_j^{\bullet})=0.$
    \item $\mathcal{E}\text{xt}^i_{\mathcal{O}_X}(\mathcal{H}^p(C_j^{\bullet}),\omega_X)=0$ for all $i\geq 1$.
    \item The maps $g_j, h_j$ described in \lemmaref{lcdlemma3} (2) are surjective; in particular $f_j=h_j\circ g_j$ is also surjective.
    \item $R^{q_X-1}\sigma_*\Omega_{\mathbb{P}}^{N-j}(\log E)=0$.
\end{enumerate}
\end{proposition}

\begin{proof}
Fix $i,j\geq 1$. Recall that 
$H^p(\mathcal{C}_j^{\bullet})=0$ for all $p<0$. Moreover, passing to the cohomology of the distinguished triangle \eqref{cone-exact}, we obtain the exact sequence
\begin{equation}\label{lcdex0}
    0\to \mathcal{H}^0(\mathcal{C}_j^{\bullet})\to \Omega_X^j\to q_*\Omega_{\Phi}^j\to \mathcal{H}^1(\mathcal{C}_j^{\bullet})\to 0
\end{equation}
and the isomorphisms
\begin{equation}
    \mathcal{H}^p(\mathcal{C}_j^{\bullet})\cong R^{p-1}q_*\Omega_{\Phi}^j\,\textrm{ for all }\, p\geq 2,
\end{equation}
whence by \theoremref{mainfiltthm} we obtain 
$$\mathcal{E}\textrm{\textit{xt}}^i_{\mathcal{O}_X}(\mathcal{H}^p(\mathcal{C}_j^{\bullet}),\omega_X)=0\,\textrm{ for all }\, p\geq 2.$$
Let $(\Omega_{\Phi}^j, {}_jL^{\bullet}) $ be the filtration coming from \lemmaref{compatibility} via the exact sequence \eqref{gen1ex}:
$${}_jL^0=\Omega_{\Phi}^j\supseteq {}_jL^1\supseteq \cdots\supseteq {}_jL^j\supseteq {}_jL^{j+1}=0\,\textrm{ with } \textrm{Gr}_{{}_jL}^k(\Omega_{\Phi}^j)\cong q^*\Omega_X^k\otimes \Omega_{\Phi/X}^{j-k}.$$
In particular, $q_*q^*\Omega_X^j\cong \Omega_X^j$ (apply projection formula and combine it with \eqref{cs}) injects inside $q_*\Omega_{\Phi}^j$ whence (1) follows from \eqref{lcdex0}. Consequently \eqref{lcdex0} becomes the exact sequence
\begin{equation}\label{lcdex0'}
    0\to\Omega_X^j\to q_*\Omega_{\Phi}^j\to \mathcal{H}^1(\mathcal{C}_j^{\bullet})\to 0
\end{equation}
Notice that (3) and (4) follows from (2) through \lemmaref{lcdlemma3} and \lemmaref{lcdlemma4}. Thus, we only need to prove (2).

Another application of \theoremref{mainfiltthm} now shows via \eqref{lcdex0'}
$$\mathcal{E}\textrm{\textit{xt}}^i_{\mathcal{O}_X}(\mathcal{H}^1(\mathcal{C}_j^{\bullet}),\omega_X)=0\,\textrm{ for all }\, i\geq 2,$$
and gives us the exact sequence 
\small 
\begin{equation}\label{random'}
    0\to \mathcal{H}\textrm{\textit{om}}_{\mathcal{O}_X}(\mathcal{H}^1(\mathcal{C}_j^{\bullet}),\omega_X)\to \mathcal{H}\textrm{\textit{om}}_{\mathcal{O}_X}(q_*\Omega_{\Phi}^j,\omega_X)\to \mathcal{H}\textrm{\textit{om}}_{\mathcal{O}_X}(\Omega_X^j,\omega_X)\to \mathcal{E}\textrm{\textit{xt}}^1_{\mathcal{O}_X}(\mathcal{H}^1(\mathcal{C}_j^{\bullet}),\omega_X)\to 0.
\end{equation}
\normalsize 
It remains to prove that $$\mathcal{E}\textrm{\textit{xt}}^1_{\mathcal{O}_X}(\mathcal{H}^1(\mathcal{C}_j^{\bullet}),\omega_X)=0.$$
\begin{claim}\label{lcdclaim1}
The maps $q_*\Omega_{\Phi}^s\to q_*\Omega_{\Phi/X}^s$ arising from \eqref{gen1ex} are surjective for all $s$. 
\end{claim}
\begin{proof}
Clearly the assertion holds for $s=0$ whence we assume $s\geq 1$. Notice that $h^0(\Omega_{\Phi}^s|_{F_x})$ is independent of $x\in X$ by \theoremref{mainfiltthm}, and so is $h^0(\Omega_{\Phi/X}^s|_{F_x})=h^0(\Omega_{F_x}^s)$, so by Grauert's theorem (\cite[Chapter III, Corollary 12.9]{Har}), we have isomorphisms:
$$q_*\Omega_{\Phi}^s\otimes \mathbb{C}(x)\cong H^0(\Omega_{\Phi}^s|_{F_x})\,\textrm{ and }\, q_*\Omega_{\Phi/X}^s\otimes \mathbb{C}(x)\cong H^0(\Omega_{F_x}^s)$$ for all $x\in X$. Thus, it is enough to show that the maps $$H^0(\Omega_{\Phi}^s|_{F_x})\to H^0(\Omega_{F_x}^s)$$ induced by \eqref{filt1} is surjective. But this is a consequence of \claimref{random}.
\end{proof}
\begin{claim}\label{lcdclaim2}
The maps $q_*({}_jL^k)\to q_*(\textrm{Gr}_{{}_jL}^k(\Omega_{\Phi}^j)$ induced by the filtration $(\Omega_{\Phi}^j, {}_jL^{\bullet})$ is surjective for all $k$.
\end{claim}
\begin{proof}
As in the proof of \lemmaref{compatibility}, we work with the commutative diagram:
\[
\begin{tikzcd}
0\arrow[r] & q^*\Omega_X^k\otimes{}_{j-k}L^1\arrow[r]\arrow[d] & q^*\Omega_X^k\otimes \Omega_{\Phi}^{j-k}\arrow[r]\arrow[d] & q^*\Omega_X^k\otimes \Omega_{\Phi/X}^{j-k}\arrow[r]\arrow[d, equal] & 0\\
0\arrow[r] & {}_jL^{k+1}\arrow[r] & {}_jL^k\arrow[r] & \textrm{Gr}^k_{{}_jL}(\Omega_{\Phi}^j)\arrow[r] & 0
\end{tikzcd}
\]
Thus, it is enough to show that the maps $$q_*\Omega_{\Phi}^{j-k}\to q_*\Omega_{\Phi/X}^{j-k}$$ are surjective, which follow from \claimref{lcdclaim1}.
\end{proof}
Thanks to \claimref{lcdclaim2}, we obtain the exact sequences 
\begin{equation}\label{lcdex1}
0\to q_*({}_jL^{k+1})\to q_*({}_jL^k)\to q_*(\textrm{Gr}_{{}_jL}^k(\Omega_{\Phi}^j))\to 0
\end{equation}
for all $k$. Now, 
it is enough to show that $$\mathcal{H}\textrm{\textit{om}}_{\mathcal{O}_X}(q_*({}_jL^{k}),\omega_X)\to \mathcal{H}\textrm{\textit{om}}_{\mathcal{O}_X}(q_*({}_jL^{k+1}),\omega_X)$$ are surjective for all $k$, as $\mathcal{E}\textrm{\textit{xt}}_{\mathcal{O}_X}^1(\mathcal{H}^1(\mathcal{C}_j^{\bullet}),\omega_X)$ is the cokernel of the composition 
\small 
$$\mathcal{H}\textrm{\textit{om}}_{\mathcal{O}_X}(q_*({}_jL^{0}),\omega_X)\to\mathcal{H}\textrm{\textit{om}}_{\mathcal{O}_X}(q_*({}_jL^{1}),\omega_X)\to\cdots \to\mathcal{H}\textrm{\textit{om}}_{\mathcal{O}_X}(q_*({}_jL^{j-1}),\omega_X)\to \mathcal{H}\textrm{\textit{om}}_{\mathcal{O}_X}(q_*({}_jL^{j}),\omega_X)$$
\normalsize
by \eqref{random'}.
But this follows immediately from \eqref{lcdex1} as the sheaves $q_*(\textrm{Gr}_{{}_jL}^k(\Omega_{\Phi}^j))$ are locally free by projection formula, \eqref{cs} and Grauert's theorem. 
\end{proof}

\begin{lemma}\label{lcdlemma6}
Let $j\geq 1$. The map $g_j:R^nt_*\Omega^{2n+1-j}_{\mathbb{P}(\mathcal{E}_L)}(\log\Phi)\to R^nq_*\Omega_{\Phi}^{2n-j}$ of \lemmaref{lcdlemma3} (2) can be expressed as $g_j=g_j''\circ g_j'$ where $g_j'$ and $g_j''$ are as follows:
\begin{itemize}
        \item $g_j':R^nt_*\Omega^{2n+1-j}_{\mathbb{P}(\mathcal{E}_L)}(\log\Phi)\to \mathcal{H}^{n+1}({\bf D}_{\Sigma}({\bf R}q_*\Omega_{\mathbb{P}(\mathcal{E}_L)}^{j}(\log\Phi)|_{\Phi}))$ arising from the distinguished triangle
        $${\bf D}_{\Sigma}({\bf R}q_*\Omega_{\mathbb{P}(\mathcal{E}_L)}^j(\log\Phi)|_{\Phi})\to {\bf D}_{\Sigma}({\bf R}t_*\Omega_{\mathbb{P}(\mathcal{E}_L)}^j(\log\Phi))\to {\bf R}t_*\Omega^{2n+1-j}_{\mathbb{P}(\mathcal{E}_L)}(\log\Phi)\xrightarrow{+1}.$$
        \item $g_j'': \mathcal{H}^{n+1}({\bf D}_{\Sigma}({\bf R}q_*\Omega_{\mathbb{P}(\mathcal{E}_L)}^{j}(\log\Phi)|_{\Phi}))\to R^nq_*\Omega_{\Phi}^{2n-j}$ arising from the distinguished triangle 
        $$ {\bf D}_{\Sigma}({\bf R}q_*\Omega_{\Phi}^{j-1})\to {\bf D}_{\Sigma}({\bf R}q_*\Omega_{\mathbb{P}(\mathcal{E}_L)}^j(\log\Phi)|_{\Phi})\to {\bf D}_{\Sigma}({\bf R}q_*\Omega_{\Phi}^{j})\xrightarrow{+1}$$ via \lemmaref{lcdlemma3} (1).
\end{itemize}
In particular, if $L$ satisfies $(Q_s)$ property with $s=\min\left\{2n+1-j,n\right\}$,  then $g_j'$ is an isomorphism. 
\end{lemma}
\begin{proof} Follows immediately from the following commutative diagram with distinguished left two columns and bottom two rows (triangles) obtained from \eqref{omega}:
\begin{equation}\label{lcddiag}
\begin{tikzcd}
{\bf D}_{\Sigma}({\bf R}q_*\Omega_{\Phi}^{j-1}) \arrow[r, equal]\arrow[d] & {\bf D}_{\Sigma}({\bf R}q_*\Omega_{\Phi}^{j-1})\arrow[d] & &\\
{\bf D}_{\Sigma}({\bf R}q_*\Omega_{\mathbb{P}(\mathcal{E}_L)}^j(\log\Phi)|_{\Phi})\arrow[r]\arrow[d] & {\bf D}_{\Sigma}({\bf R}t_*\Omega_{\mathbb{P}(\mathcal{E}_L)}^j(\log\Phi))\arrow[d]\arrow[r] & {\bf R}t_*\Omega^{2n+1-j}_{\mathbb{P}(\mathcal{E}_L)}(\log\Phi)\arrow[d, equal]\arrow[r, "+1"] & \textrm{ } \\
{\bf D}_{\Sigma}({\bf R}q_*\Omega_{\Phi}^{j}) \arrow[r]\arrow[d, "+1"] & {\bf D}_{\Sigma}({\bf R}t_*\Omega_{\mathbb{P}(\mathcal{E}_L)}^{j}) \arrow[r]\arrow[d, "+1"] & {\bf R}t_*\Omega^{2n+1-j}_{\mathbb{P}(\mathcal{E}_L)}(\log\Phi) \arrow[r, "+1"] & \textrm{ }\\
\textrm{ } & \textrm{ } & & 
\end{tikzcd}
\end{equation}
The last assertion follows since $\mathcal{H}^i({\bf D}_{\Sigma}({\bf R}t_*\Omega_{\mathbb{P}(\mathcal{E}_L)}^j(\log\Phi)))\cong R^it_*\Omega_{\mathbb{P}(\mathcal{E}_L)}^{2n+1-j}(\log\Phi)(-\Phi)=0$ for $i\geq 1$ by \propositionref{rivanishing}.
\end{proof}

\begin{proposition}\label{updated}
Assume $L$ satisfies $(Q_n)$-property, and let $j\geq 1$ be an integer. Then 
\begin{equation}\label{hilcd}
    R^{q_X-2}\sigma_*\Omega_{\mathbb{P}}^{N-j}(\log E)=0
\end{equation}
if and only if $H^i(\mathcal{O}_X)=0$ for all $1\leq i\leq j$.
\end{proposition}
\begin{proof}
We first prove the 

\begin{claim}\label{extraclaim} The following conditions on $h_j: R^nq_*\Omega_{\Phi}^{2n-j}\to \Omega_X^{n-j}$ described in \lemmaref{lcdlemma3} (2) are equivalent:
\begin{enumerate}
    \item $h_j$ is injective,
    \item $h_j$ is an isomorphism,
    \item $H^i(\mathcal{O}_X)=0$ for all $1\leq i\leq j$.
\end{enumerate}
\end{claim}
\begin{proof}
The equivalence of (1) and (2) follows from the surjectivity of $h_j$ proven in \propositionref{lcdcor} (3). For the same reason, (2) holds if and only if $$\textrm{rank}(R^nq_*\Omega_{\Phi}^{2n-j})=\textrm{rank}(\Omega_X^{n-j}).$$ Thus, it is enough to prove that the above equality is equivalent to (3). To this end, apply Theorem \ref{mainfiltthm} (1) to obtain 
\begin{equation*}\label{lcdgen1}
    \textrm{rank}(R^nq_*\Omega_{\Phi}^{2n-j})=\binom{n}{n-j}+\sum_{k=0}^{j-1}\binom{n}{k}h^{0,j-k}(X).
\end{equation*}
On the other hand, we have 
\begin{equation*}\label{lcdgen2}
    \textrm{rank}(\Omega_X^{n-j})=\binom{n}{n-j},
\end{equation*}
whence the assertion follows.
\end{proof}

Observe that by \lemmaref{lcdlemma2} (2), \eqref{hilcd} holds if and only if the map $f_j$ is injective.

First assume \eqref{hilcd} holds, 
i.e. $f_j$ is injective. 
Now, by \lemmaref{lcdlemma3} (2) and (3), $f_j=h_j\circ g_j$ and $g_j$ is surjective. Thus, $h_j: R^nq_*\Omega_{\Phi}^{2n-j}\to \Omega_X^{n-j}$ is injective, whence $H^i(\mathcal{O}_X)=0$ for all $1\leq i\leq j$ by \claimref{extraclaim}. 

Conversely, assume $H^i(\mathcal{O}_X)=0$ for all $1\leq i\leq j$, and we aim to show that $f_j$ is an injection. By \claimref{extraclaim}, $h_j$ is an isomorphism. Also, by \lemmaref{lcdlemma6}, we have $g_j=g_j''\circ g_j'$ and $g_j'$ is an isomorphism. Thus, it is enough to show that $g_j''$ is an injection, or equivalently the map $$r_j: \mathcal{H}^n({\bf D}_{\Sigma}({\bf R}q_*\Omega_{\Phi}^{j}))\cong R^{n-1}q_*\Omega_{\Phi}^{2n-j}\to \mathcal{H}^{n+1}({\bf D}_{\Sigma}({\bf R}q_*\Omega_{\Phi}^{j-1}))\cong R^nq_*\Omega_{\Phi}^{2n+1-j}$$ is surjective (the isomorphisms above are again consequence of \lemmaref{lcdlemma3} (1)). Now, evidently the map $r_j$ is
arising from the exact sequence 
$$0\to \Omega_{\Phi}^{2n+1-j}\to \Omega_{\mathbb{P}(\mathcal{E}_L)}^{2n+1-j}(\log\Phi)|_{\Phi}\to \Omega_{\Phi}^{2n-j}\to0.$$ By \theoremref{mainfiltthm} (2), we have $$R^{n-1}q_*\Omega_{\Phi}^{2n-j}\otimes\mathbb{C}(x)\cong H^{n-1}(\Omega_{\Phi}^{2n-j}|_{F_x})\,\textrm{ and }\, R^nq_*\Omega_{\Phi}^{2n+1-j}\otimes \mathbb{C}(x)\cong H^n(\Omega^{2n+1-j}_{\Phi}|_{F_x}) $$ for all $x\in X$. Thus, it is enough to show that the map $$H^{n-1}(\Omega_{\Phi}^{2n-j}|_{F_x})\to H^n(\Omega^{2n+1-j}_{\Phi}|_{F_x})$$ is surjective, or equivalently its dual $$H^0(\Omega_{\Phi}^{j-1}|_{F_x})\to H^1(\Omega_{\Phi}^{j}|_{F_x})$$ is injective. Now, this map is the cup product by the image of $c_1(\Phi)\in H^1(\Omega^1_{\Phi})$ through the restriction map $H^1(\Omega^1_{\Phi})\to H^1(\Omega^1_{\Phi}|_{F_x})$. Recall from \eqref{another} that 
$$\Omega_{E_{\Delta}|_{E_x}}^1\cong \mathcal{O}_{E_x}^{\oplus n}\oplus \Omega_{E_x}^1.$$
To this end, we use the commutative diagram (the upward vertical maps are injective because of the splitting)
\[
\begin{tikzcd}
H^0(\Omega_{\Phi}^{j-1}|_{F_x})\arrow[r]\arrow[d, two heads] & H^1(\Omega_{\Phi}^{j}|_{F_x})\arrow[d, two heads]\\
H^0(\Omega_{E_{\Delta}}^{j-1}|_{E_x})\arrow[r] & H^1(\Omega_{E_{\Delta}}^{j}|_{E_x})\\
H^0(\wedge^{j-1}\mathcal{O}_{E_x}^{\oplus n})\arrow[u, hook] \arrow[r] & H^1(\wedge^{j-1}\mathcal{O}_{E_x}^{\oplus n}\otimes \Omega_{E_x}^1)\arrow[u, hook]
\end{tikzcd}
\]
The downward vertical maps are surjections by \lemmaref{filtlemma3} (the surjection of the left one corresponding to $j=1$ is obvious). Once again, by \theoremref{mainfiltthm} (1) and \lemmaref{filtlemma} (5), we have $$h^0(\Omega_{\Phi}^{j-1}|_{F_x})=h^0(\Omega_{E_{\Delta}}^{j-1}|_{E_x})=\binom{n}{j-1}$$ as $h^i(\mathcal{O}_{X})=0$ for $1\leq i\leq j-1$, whence the downward left vertical map is an isomorphism. The upward left vertical map is also an isomorphism for dimension reason. Finally the bottom horizontal map is an injection as it is a coordinate-wise map induced by $$H^0(\mathcal{O}_{E_x})\to H^1(\Omega_{E_x}^1)$$ sending $1\in H^0(\mathcal{O}_{E_x})$ to a non cohomologically trivial class $c_1(\Phi|_{E_x})\in H^{1,1}(E_x)$ (see \eqref{further}). Thus the middle horizontal map is an injection, and consequently so is the top horizontal map. The proof is now complete.
\end{proof}

\begin{proposition}\label{lcdnv}
Assume $n\geq 2$ and $L$ satisfies $(Q_n)$-property. Then $R^{q_X-3}\sigma_*\Omega_{\mathbb{P}}^{N-1}(\log E)\neq 0$.
\end{proposition}
\begin{proof}
First recall from \lemmaref{lcdlemma2} (1) that $$R^{q_X-3}\sigma_*\Omega_{\mathbb{P}}^{N-1}(\log E)\cong R^{n-1}t_*\Omega_{\mathbb{P}(\mathcal{E}_L)}^{2n}(\log\Phi).$$
To this end, we work with the exact sequence (middle row of \eqref{omega})
$$0\to \Omega_{\mathbb{P}(\mathcal{E}_L)}^{2n}(\log \Phi)(-\Phi)\to \Omega_{\mathbb{P}(\mathcal{E}_L)}^{2n}(\log \Phi)\to \Omega_{\mathbb{P}(\mathcal{E}_L)}^{2n}(\log \Phi)|_{\Phi}\to 0$$ and use the fact that $R^it_* \Omega_{\mathbb{P}(\mathcal{E}_L)}^{2n}(\log \Phi)(-\Phi)=0$ for $i\geq 1$ (see \propositionref{rivanishing}) to conclude that $$R^{n-1}t_*\Omega_{\mathbb{P}(\mathcal{E}_L)}^{2n}(\log\Phi)\cong R^{n-1}t_*\Omega_{\mathbb{P}(\mathcal{E}_L)}^{2n}(\log\Phi)|_{\Phi}.$$ For the sake of contradiction, assume $R^{q_X-3}\sigma_*\Omega_{\mathbb{P}}^{N-1}(\log E)=0$ whence $ R^{n-1}t_*\Omega_{\mathbb{P}(\mathcal{E}_L)}^{2n}(\log\Phi)|_{\Phi}=0$. This, via the exact sequence (right column of \eqref{omega}) $$0\to\omega_{\Phi}\to \Omega_{\mathbb{P}(\mathcal{E}_L)}^{2n}(\log\Phi)|_{\Phi}\to \Omega_{\Phi}^{2n-1}\to 0$$ yields that the map  $$R^{n-1}q_*\Omega_{\Phi}^{2n-1}\to R^nq_*\omega_{\Phi}$$ is an injection. But according to \theoremref{mainfiltthm} (1), we have $$\textrm{rank}(R^{n-1}q_*\Omega_{\Phi}^{2n-1})=nh^{1,0}(X)+h^{n-1,n-1}(X)+1\geq 2\,\textrm{ and }\, \textrm{rank}(R^{n}q_*\omega_{\Phi})=1,$$ (the first equality is where we use $n\geq 2$) a contradiction.
\end{proof}

We are now ready to provide the 

\begin{proof}[Proof of \theoremref{mainlcd}] To start with, note that $F_i\mathcal{H}^j_{\Sigma}(\mathcal{O}_{\mathbb{P}^N})=0$ for all $i<0,j\geq 0$ by \cite[Remark 3.4]{MP}. It follows that 
\begin{equation}\label{again2}
    \text{gl}(F_{\bullet}\mathcal{H}^{\mathrm{lcd}(\mathbb{P}^N,\Sigma)}_{\Sigma}(\mathcal{O}_{\mathbb{P}^N}))\geq 0.
\end{equation}
Next, we need the following 
\begin{claim}\label{again3}
Assume $L$ satisfies $(Q_0)$-property. Then $q_X-\nu(X)-2\leq \mathrm{lcd}(\mathbb{P}^N,\Sigma)\leq q_X-2$.
\end{claim}
\begin{proof}
Recall that $\Sigma$ has Du Bois singularities under our assumption (see \theoremref{previous} (1)). Consequently $\textrm{lcd}({\mathbb{P}^N},\Sigma)\geq q_X-\nu(X)-2=N-\textrm{depth}(\mathcal{O}_{\Sigma})$ by \cite[Remark 11.7]{MP} and \theoremref{depth}.
For the upper bound of $\mathrm{lcd}(\mathbb{P}^N,\Sigma)$, according to \cite[Theorem E]{MP}, we need to show that 
\begin{equation}\label{needlcd}
    R^{j'+i}\sigma_*\Omega_{\mathbb{P}}^{N-i}(\log E)=0\,\textrm{ for all } j'\geq q_X-2\,\textrm{ and }\, i\geq 0.
\end{equation} 
To this end, recall that $R^k\sigma_*\Omega_{\mathbb{P}}^{N-i}(\log E)=0$ for all $k\geq q_X$ by \lemmaref{lcdlemma2} (3). 
Thanks to \propositionref{freelcd}, it only remains to verify that 
\begin{equation}
    R^{q_X-1}\sigma_*\Omega_{\mathbb{P}}^{N-1}(\log E)=0
\end{equation}
in order to prove \eqref{needlcd}. The required vanishing follows from \propositionref{lcdcor} (4).
\end{proof}

We now split the proof into two cases.

\smallskip

\noindent{\bf Case $1$: $\nu(X)=0$.}  It follows from \claimref{again3} that in this case 
$$\textrm{lcd}(\mathcal{O}_{\mathbb{P}^N},\Sigma)=q_X-2.$$
Now, according to \cite[Theorem 10.2]{MP}, the Hodge filtration on $\mathcal{H}^{q_X-2}_{\Sigma}(\mathcal{O}_{\mathbb{P}^N})$ is generated at level $1$ if and only if $$R^{q_X-3+j}\sigma_*\Omega_{\mathbb{P}}^{N-j}(\log E)=0\,\textrm{ for all }\, j\geq 2.$$ Since the required vanishings hold for $j\geq 3$ (see \lemmaref{lcdlemma2} (3)), the only non-trivial case is the one corresponding to case $j=2$. But this vanishing follows from \propositionref{lcdcor} (4). Also, the Hodge filtration is generated at level $0$ if and only if $$R^{q_X-2}\sigma_*\Omega_{\mathbb{P}}^{N-1}(\log E)=0,$$ and this is equivalent to $H^1(\mathcal{O}_X)=0$ by \propositionref{updated}. The assertion (1) now follows from \eqref{again2} and the fact that $h^1(\mathcal{O}_X)=\nu(X)=0\implies X\cong\mathbb{P}^1$.

\smallskip

\noindent{\bf Case $2$: $\nu(X)\geq 1$.} In this case, note that $n\geq 2$. We again apply \cite[Theorem E]{MP}. To see that $\textrm{lcd}(\mathbb{P}^N,\Sigma)\leq q_X-3$, we need to show that $$R^{c+j}\sigma_*\Omega_{\mathbb{P}}^{N-j}(\log E)=0\,\textrm{ for all }\, c\geq q_X-3, j\geq 0.$$ This follows immediately from \propositionref{freelcd}, \propositionref{lcdcor} (4), \propositionref{updated}, and \lemmaref{lcdlemma2} (3). Also, $\textrm{lcd}(\mathbb{P}^N,\Sigma)\geq q_X-3$ by \propositionref{lcdnv}. To calculate the generation level of the Hodge filtration, we again use \cite[Theorem 10.2]{MP}. Accordingly, to see that the Hodge filtration on $\mathcal{H}^{q_X-3}_{{\Sigma}}(\mathcal{O}_{\mathbb{P}^N})$ is generated at level $2$, we need to show that $$R^{q_X-4+j}\sigma_*\Omega_{\mathbb{P}}^{N-j}(\log E)=0\,\textrm{ for all }\, j\geq 3$$ which follows from \propositionref{lcdcor} (4) and \lemmaref{lcdlemma2} (3). Now, the Hodge filtration is generated at level $1$ if and only if
$$R^{q_X-2}\sigma_*\Omega_{\mathbb{P}}^{N-2}(\log E)=0$$ which holds if and only if $H^i(\mathcal{O}_X)=0$ for $i=1,2$ by \propositionref{updated}.
Finally, it is not generated at level $0$ by \propositionref{lcdnv}. Consequently (2) follows. 
\end{proof}

\begin{remark}
When $n=1$, the inequality $\textrm{lcd}(\mathbb{P}^N,\Sigma)\leq q_X-2=N-3$ can also be deduced from Dao-Takagi-Varbaro theorem (\cite[Theorem 11.21]{MP}) since $\textrm{depth}(\mathcal{O}_{\Sigma})=3$ by \theoremref{depth}.
\end{remark}

We now provide the proofs of the corollaries: 

\begin{proof}[Proof of \corollaryref{kodaira}]
(1)  is an immediate consequence of \theoremref{mainlcd} and \cite[Theorem 2.16]{PS}. (2) follows from (1) and \cite[Theorem A and Theorem B]{ORS}.
\end{proof}

\begin{proof}[Proof of \corollaryref{corb}]
By \propositionref{per}, $\mathbb{Q}_{\Sigma}[2n+1]$ is perverse if and only if $\textrm{lcd}(\mathbb{P}^N,\Sigma)=q_X-n-1$, whence the assertion follows from \theoremref{mainlcd}.
\end{proof}

\begin{proof}[Proof of \corollaryref{corf}] If $\Sigma$ is lci, then $\mathbb{Q}_{\Sigma}[2n+1]$ is perverse, whence $n\leq 2$ by \corollaryref{corb}. Thus either (i) or (ii) holds by \theoremref{maingor}. For the partial converse, the only non-trivial case is that of an elliptic normal curve of degree six, $\Sigma$ in this case is a complete intersection by \cite[Remark 5.6]{ENP}.
\end{proof}

We end this article by providing equivalent characterizations of the secant varieties of rational normal curves, that we obtain by combining our work with works of other authors (see \cite{CK} for the relevant notation): 

\begin{corollary}\label{ratnorm}
Let $X$ be a smooth projective curve of genus $g$ and let $L$ be a line bundle on $X$ with $\mathrm{deg}(L)\geq 2g+3$. Assume $\Sigma\neq\mathbb{P}^N$. Then the following are equivalent:
\begin{enumerate}
    \item $(X,L)\cong(\mathbb{P}^1,\mathcal{O}_{\mathbb{P}^1}(d))$ with $d\geq 4$,
    \item $\mathrm{deg}(\Sigma)=\binom{N-1}{2}$ i.e. $\Sigma$ is a 2-secant variety of minimal degree,
    \item $\Sigma$ has 2-pure Cohen-Macaulay Betti table,
    \item $\beta_{p,2}=\binom{p+1}{2}\binom{N-1}{p+2}$ for all $p$,
    \item $\dim(I_{\Sigma})_3=\binom{N-1}{3}$,
    \item $\beta_{p,2}=\binom{p+1}{2}\binom{N-1}{p+2}$ for some $1\leq p\leq N-3$,
    \item the third strand of $\Sigma$ has length $N-3$,
    \item $\mathrm{reg}(\Sigma)=3$,
    \item $\Sigma$ satisfies $(N_{3,N-3})$-property,
    \item $\Sigma$ is a Fano variety with log terminal singularities,
    \item $\Sigma$ has rational singularities,
    \item $\Sigma$ has quotient singularities,
    \item the singularities of $\Sigma$ are pre-$1$-rational,
    \item the singularities of $\Sigma$ are pre-$p$-rational for all $p$,
    \item $\mathrm{gl}(F_{\bullet}\mathcal{H}^{\mathrm{lcd}(\mathbb{P}^N,\Sigma)}_{{\Sigma}}(\mathcal{O}_{\mathbb{P}^N}))=0$.
\end{enumerate}
\end{corollary}

\begin{proof} The equivalence of (1) with (12)-(15) comes from \corollaryref{cord}. (1) and (11) are equivalent by \cite[Proposition 9]{V2} (or by \theoremref{previous}(1)(d)). The equivalence of (1) and (10) follows from \cite[Theorem 1.1]{ENP}. Now, since $X$ is an irreducible curve by assumption, $\Sigma$ is a 2-secant variety of minimal degree if and only if (1) holds (see the introduction of \cite{CK}, or the results of \cite{CR}, in particular \cite[Theorem 6.1]{CR}). Thus, the equivalence of (1)-(9) follows from \cite[Theorem 1.1]{CK}.  
\end{proof}

\bibliography{bibsecant}

\end{document}

%% file: macros.tex
\usepackage{amsmath,amsfonts,amsthm,mathrsfs}
\usepackage{amssymb, stmaryrd}

\usepackage[unicode,bookmarks,pagebackref]{hyperref}
\usepackage[usenames,dvipsnames]{xcolor}
\hypersetup{colorlinks=true,citecolor=NavyBlue,linkcolor=Brown,urlcolor=Orange}

\usepackage[alphabetic,initials]{amsrefs}

\usepackage{enumitem}

\usepackage{chngcntr}

\ifdraft
\usepackage[notcite,notref,color]{showkeys}

\definecolor{labelkey}{gray}{0.5}
\fi

\usepackage{tikz}

\usepackage{tikz-cd}

\usetikzlibrary{matrix,arrows}
\newlength{\myarrowsize} 

\pgfarrowsdeclare{cmto}{cmto}{
	\pgfsetdash{}{0pt} 
	\pgfsetbeveljoin 
	\pgfsetroundcap 
	\setlength{\myarrowsize}{0.6pt}
	\addtolength{\myarrowsize}{.5\pgflinewidth}
	\pgfarrowsleftextend{-4\myarrowsize-.5\pgflinewidth} 
	\pgfarrowsrightextend{.8\pgflinewidth}
}{
	\setlength{\myarrowsize}{0.6pt} 
  	\addtolength{\myarrowsize}{.5\pgflinewidth}  
	\pgfsetlinewidth{0.5\pgflinewidth}
	\pgfsetroundjoin
	\pgfpathmoveto{\pgfpoint{1.5\pgflinewidth}{0}}
	\pgfpatharc{-109}{-170}{4\myarrowsize}
	\pgfpatharc{10}{189}{0.58\pgflinewidth and 0.2\pgflinewidth}
	\pgfpatharc{-170}{-115}{4\myarrowsize+\pgflinewidth}
	\pgfpathclose
	\pgfusepathqfillstroke
	\pgfpathmoveto{\pgfpoint{1.5\pgflinewidth}{0}}
	\pgfpatharc{109}{170}{4\myarrowsize}
	\pgfpatharc{-10}{-189}{0.58\pgflinewidth and 0.2\pgflinewidth}
	\pgfpatharc{170}{115}{4\myarrowsize+\pgflinewidth}
	\pgfpathclose
	\pgfusepathqfillstroke
	\pgfsetlinewidth{2\pgflinewidth}
}

\pgfarrowsdeclare{cmonto}{cmonto}{
	\pgfsetdash{}{0pt} 
	\pgfsetbeveljoin 
	\pgfsetroundcap 
	\setlength{\myarrowsize}{0.6pt}
	\addtolength{\myarrowsize}{.5\pgflinewidth}
	\pgfarrowsleftextend{-4\myarrowsize-.5\pgflinewidth} 
	\pgfarrowsrightextend{.8\pgflinewidth}
}{
	\setlength{\myarrowsize}{0.6pt} 
  	\addtolength{\myarrowsize}{.5\pgflinewidth}  
	\pgfsetlinewidth{0.5\pgflinewidth}
	\pgfsetroundjoin
	\pgfpathmoveto{\pgfpoint{1.5\pgflinewidth}{0}}
	\pgfpatharc{-109}{-170}{4\myarrowsize}
	\pgfpatharc{10}{189}{0.58\pgflinewidth and 0.2\pgflinewidth}
	\pgfpatharc{-170}{-115}{4\myarrowsize+\pgflinewidth}
	\pgfpathclose
	\pgfusepathqfillstroke
	\pgfpathmoveto{\pgfpoint{1.5\pgflinewidth}{0}}
	\pgfpatharc{109}{170}{4\myarrowsize}
	\pgfpatharc{-10}{-189}{0.58\pgflinewidth and 0.2\pgflinewidth}
	\pgfpatharc{170}{115}{4\myarrowsize+\pgflinewidth}
	\pgfpathclose
	\pgfusepathqfillstroke
	\pgfpathmoveto{\pgfpoint{1.5\pgflinewidth-0.3em}{0}}
	\pgfpatharc{-109}{-170}{4\myarrowsize}
	\pgfpatharc{10}{189}{0.58\pgflinewidth and 0.2\pgflinewidth}
	\pgfpatharc{-170}{-115}{4\myarrowsize+\pgflinewidth}
	\pgfpathclose
	\pgfusepathqfillstroke
	\pgfpathmoveto{\pgfpoint{1.5\pgflinewidth-0.3em}{0}}
	\pgfpatharc{109}{170}{4\myarrowsize}
	\pgfpatharc{-10}{-189}{0.58\pgflinewidth and 0.2\pgflinewidth}
	\pgfpatharc{170}{115}{4\myarrowsize+\pgflinewidth}
	\pgfpathclose
	\pgfusepathqfillstroke
	\pgfsetlinewidth{2\pgflinewidth}
}

\pgfarrowsdeclare{cmhook}{cmhook}{
	\pgfsetdash{}{0pt} 
	\pgfsetbeveljoin 
	\pgfsetroundcap 
	\setlength{\myarrowsize}{0.6pt}
	\addtolength{\myarrowsize}{.5\pgflinewidth}
	\pgfarrowsleftextend{-4\myarrowsize-.5\pgflinewidth} 
	\pgfarrowsrightextend{.8\pgflinewidth}
}{
	\setlength{\myarrowsize}{0.6pt} 
  	\addtolength{\myarrowsize}{.5\pgflinewidth}  
 	\pgfsetdash{}{0pt}
	\pgfsetroundcap
	\pgfpathmoveto{\pgfqpoint{0pt}{-4.667\pgflinewidth}}
	\pgfpathcurveto
    {\pgfqpoint{4\pgflinewidth}{-4.667\pgflinewidth}}
    {\pgfqpoint{4\pgflinewidth}{0pt}}
    {\pgfpointorigin}
	\pgfusepathqstroke
}

\newenvironment{diagram*}[2]{%
\[%
\begin{tikzpicture}[>=cmto,baseline=(current bounding box.center),%
	to/.style={->,font=\scriptsize,cap=round},%
	into/.style={cmhook->,font=\scriptsize,cap=round},%
	onto/.style={-cmonto,font=\scriptsize,cap=round},%
	math/.style={matrix of math nodes, row sep=#2, column sep=#1,%
		text height=1.5ex, text depth=0.25ex}]%
}{%
\end{tikzpicture}%
\]%
\ignorespacesafterend%
}

\newcommand{\cohH}{\mathcal{H}}

\newcommand{\ZZ}{\mathbb{Z}}
\newcommand{\QQ}{\mathbb{Q}}

\def\overbar#1#2#3{{%
	\setbox0=\hbox{\displaystyle{#1}}%
	\dimen0=\wd0
	\advance\dimen0 by -#2 
	\vbox {\nointerlineskip \moveright #3 \vbox{\hrule height 0.3pt width \dimen0}%
		\nointerlineskip \vskip 1.5pt \box0}%
}}

\makeatletter
\let\@@seccntformat\@seccntformat
\renewcommand*{\@seccntformat}[1]{%
  \expandafter\ifx\csname @seccntformat@#1\endcsname\relax
    \expandafter\@@seccntformat
  \else
    \expandafter
      \csname @seccntformat@#1\expandafter\endcsname
  \fi
    {#1}%
}
\newcommand*{\@seccntformat@subsection}[1]{%
  \textbf{\csname the#1\endcsname.}
}
\makeatother

\makeatletter
\let\@paragraph\paragraph
\renewcommand*{\paragraph}[1]{%
	\vspace{0.3\baselineskip}%
	\@paragraph{\textit{#1}}%
}
\makeatother

\counterwithin{equation}{subsection}
\counterwithout{subsection}{section}
\counterwithin{figure}{subsection}

\newtheorem{theorem}[equation]{Theorem}
\newtheorem*{theorem*}{Theorem}
\newtheorem{lemma}[equation]{Lemma}
\newtheorem*{lemma*}{Lemma}
\newtheorem{corollary}[equation]{Corollary}
\newtheorem{proposition}[equation]{Proposition}
\newtheorem*{proposition*}{Proposition}

\newtheorem{claim}[equation]{Claim}
\theoremstyle{definition}
\newtheorem{definition}[equation]{Definition}
\newtheorem*{definition*}{Definition}
\newtheorem{remark}[equation]{Remark}

\newtheorem{example}[equation]{Example}
\newtheorem*{example*}{Example}
\newtheorem*{problem*}{Problem}

\theoremstyle{plain}

\newcommand{\theoremref}[1]{\hyperref[#1]{Theorem~\ref*{#1}}}
\newcommand{\lemmaref}[1]{\hyperref[#1]{Lemma~\ref*{#1}}}
\newcommand{\definitionref}[1]{\hyperref[#1]{Definition~\ref*{#1}}}
\newcommand{\propositionref}[1]{\hyperref[#1]{Proposition~\ref*{#1}}}
\newcommand{\conjectureref}[1]{\hyperref[#1]{Conjecture~\ref*{#1}}}
\newcommand{\corollaryref}[1]{\hyperref[#1]{Corollary~\ref*{#1}}}
\newcommand{\exampleref}[1]{\hyperref[#1]{Example~\ref*{#1}}}
\newcommand{\setupref}[1]{\hyperref[#1]{Set-up~\ref*{#1}}}
\newcommand{\remarkref}[1]{\hyperref[#1]{Remark~\ref*{#1}}}
\newcommand{\claimref}[1]{\hyperref[#1]{Claim~\ref*{#1}}}
\newcommand{\figureref}[1]{\hyperref[#1]{Figure~\ref*{#1}}}

\makeatletter
\let\old@caption\caption
\renewcommand*{\caption}[1]{%
	\setcounter{figure}{\value{equation}}%
	\stepcounter{equation}%
	\old@caption{#1}\relax%
}
\makeatother

\newcounter{intro}

\newtheorem{intro-conjecture}[intro]{Conjecture}
\newtheorem{intro-corollary}[intro]{Corollary}
\newtheorem{intro-theorem}[intro]{Theorem}

\newcommand{\parref}[1]{\hyperref[#1]{\S\ref*{#1}}}

\makeatletter
\newcommand*\if@single[3]{%
  \setbox0\hbox{${\mathaccent"0362{#1}}^H$}%
  \setbox2\hbox{${\mathaccent"0362{\kern0pt#1}}^H$}%
  \ifdim\ht0=\ht2 #3\else #2\fi
  }
\newcommand*\rel@kern[1]{\kern#1\dimexpr\macc@kerna}
\newcommand*\widebar[1]{\@ifnextchar^{{\wide@bar{#1}{0}}}{\wide@bar{#1}{1}}}
\newcommand*\wide@bar[2]{\if@single{#1}{\wide@bar@{#1}{#2}{1}}{\wide@bar@{#1}{#2}{2}}}
\newcommand*\wide@bar@[3]{%
  \begingroup
  \def\mathaccent##1##2{%
    \if#32 \let\macc@nucleus\first@char \fi
    \setbox\z@\hbox{$\macc@style{\macc@nucleus}_{}$}%
    \setbox\tw@\hbox{$\macc@style{\macc@nucleus}{}_{}$}%
    \dimen@\wd\tw@
    \advance\dimen@-\wd\z@
    \divide\dimen@ 3
    \@tempdima\wd\tw@
    \advance\@tempdima-\scriptspace
    \divide\@tempdima 10
    \advance\dimen@-\@tempdima
    \ifdim\dimen@>\z@ \dimen@0pt\fi
    \rel@kern{0.6}\kern-\dimen@
    \if#31
      \overline{\rel@kern{-0.6}\kern\dimen@\macc@nucleus\rel@kern{0.4}\kern\dimen@}%
      \advance\dimen@0.4\dimexpr\macc@kerna
      \let\final@kern#2%
      \ifdim\dimen@<\z@ \let\final@kern1\fi
      \if\final@kern1 \kern-\dimen@\fi
    \else
      \overline{\rel@kern{-0.6}\kern\dimen@#1}%
    \fi
  }%
  \macc@depth\@ne
  \let\math@bgroup\@empty \let\math@egroup\macc@set@skewchar
  \mathsurround\z@ \frozen@everymath{\mathgroup\macc@group\relax}%
  \macc@set@skewchar\relax
  \let\mathaccentV\macc@nested@a
  \if#31
    \macc@nested@a\relax111{#1}%
  \else
    \def\gobble@till@marker##1\endmarker{}%
    \futurelet\first@char\gobble@till@marker#1\endmarker
    \ifcat\noexpand\first@char A\else
      \def\first@char{}%
    \fi
    \macc@nested@a\relax111{\first@char}%
  \fi
  \endgroup
}
\makeatother